\documentclass{amsart}
\usepackage{amscd,amssymb,epsfig, epsf, graphicx,epic,eepic}

\usepackage[dvipsnames]{xcolor}

\usepackage[all]{xy}
\SelectTips{cm}{}
\allowdisplaybreaks

\numberwithin{equation}{subsection}

\newtheorem{propo}{Proposition}[section]
\newtheorem{corol}[propo]{Corollary}
\newtheorem{theor}[propo]{Theorem}
\newtheorem{lemma}[propo]{Lemma}

\theoremstyle{definition}

\theoremstyle{remark}

\newcommand{\A}{\mathcal{A}}

\newcommand{\ad}{\operatorname{ad}}

\newcommand{\Aut}{\operatorname{Aut}}

\newcommand{\Com}{\operatorname{Com}}
\newcommand{\Der}{\operatorname{der}}
\newcommand{\coder}{\operatorname{coder}}

\newcommand{\End}{\operatorname{End}}

\newcommand{\g}{\mathfrak{g}}

\newcommand{\grAlg}{\operatorname{\mathcal{GA}}}
\newcommand{\grCom}{\operatorname{\mathcal{CGA}}}
\newcommand{\Hom}{\operatorname{Hom}}
\newcommand{\id}{\operatorname{id}}

\newcommand{\Ker}{\operatorname{Ker}}
\newcommand{\kk}{\mathbb{K}}
\newcommand{\Mat}{\operatorname{Mat}}

\newcommand{\RR}{\mathbb{R}}

\newcommand{\tr}{\operatorname{tr}}
\newcommand{\ZZ}{\mathbb{Z}}

\newcommand{\bracket}[2]{\left\{#1,#2\right\}}
\newcommand{\double}[2]{\left\{\!\!\left\{#1,#2\right\}\!\!\right\}}
\newcommand{\triple}[3]{\left\{\!\!\left\{#1,#2,#3\right\}\!\!\right\}}

\let\oldmarginpar\marginpar
\renewcommand\marginpar[1]{\oldmarginpar{\footnotesize #1}}

\begin{document}

  \title[{Poisson-Gerstenhaber brackets in representation algebras}]{
  Poisson-Gerstenhaber brackets in representation algebras}

\author[Vladimir Turaev]{Vladimir Turaev}
\address{Vladimir Turaev \newline
\indent   Department of Mathematics \newline
\indent  Indiana University \newline
\indent Bloomington IN47405, USA\newline
\indent $\mathtt{vturaev@yahoo.com}$
}

\begin{abstract}
We introduce cyclic bilinear forms on coalgebras and  use them to generalize Van den Bergh's   Poisson brackets in representation algebras.
\end{abstract}

\maketitle

\section {Introduction}

Van den Bergh \cite{VdB} introduced   double Poisson algebras  as a non-commutative version of Poisson geometry, see also Crawley-Boevey \cite{Cb} for a related approach. A    double Poisson
     algebra is an algebra $A$ (possibly, non-unital) equipped with  a      double   Poisson bracket defined as a linear map $\double{-}{-}:A\otimes A \to A\otimes A$ satisfying certain axioms. A      double Poisson bracket in $A$ induces  a   Poisson structure on the variety ${\rm Rep}(A,N)$ of $N$-dimensional representations of $A$ for all $N\geq 1$.   To give a more  precise formulation, note that the pair $(A,N)$   determines a commutative    algebra  $ A_N  $ and  a homomorphism   $ t_N$ from $ A$ to the algebra $\Mat_N(A_N) $  of $N\times N$ matrices over $A_N$ characterized by the following universal  property: for any commutative algebra~$B$ and any  homomorphism  $s:A\to \Mat_N(B) $, there is a unique homomorphism  $r:A_N\to B$ such that  $s=\Mat_N(r) t_N$ (see \cite{Pr, BW}). One calls  $A_N$   the coordinate algebra of   ${\rm Rep}(A,N)$. Van den Bergh showed that a  double Poisson bracket in $A$ induces a Poisson bracket
in     $ A_N  $  for all $N\geq 1$.

The aim of this paper is to generalize Van den Bergh's   Poisson bracket   in $A_N$   to a wider  class of   commutative  algebras associated with $A$. Fix a  commutative ring $\kk$ which will be the ground ring of all modules, algebras, and coalgebras. For an algebra $A$   and   a coalgebra $M$ (possibly, non-counital), we define a commutative {\it   representation algebra}  $A_M$. It is  characterized by the   universal  property as above with   $\Mat_N (B) $ replaced by  the convolution algebra $ \Hom_\kk(M,B)$. The algebra $A_M$ is described  here  via       generators $\{a_\alpha\, \vert \,  a\in A, \alpha\in M\}$
 and certain relations. For the coalgebra  $ M=(\Mat_N(\kk))^*$ dual to $\Mat_N(\kk)$, we  recover   $ A_N$ because $$\Hom_\kk((\Mat_N(\kk))^*,B)=B\otimes_\kk \Mat_N(\kk)=\Mat_N(B).$$

With each  bilinear form $v:M\otimes M\to \kk$ on a coalgebra $M$ we associate a linear map $\widehat v: M\otimes M \to M\otimes M$. We call $v$ {\it   cyclic} if $\widehat v$ commutes with the permutation of the factors.  Our main result is the following theorem.

 \begin{theor}\label{intro} For a  double Poisson algebra $(A,  \double{-}{-})$  and a  cyclic bilinear form $v$ on a coalgebra $M$, there is a unique   Poisson bracket
 $\bracket{-}{-}_v $ in the representation algebra $A_M$ such that
 for  any $a,b\in A$, $\alpha, \beta \in M$, and any  finite expansions
  $$\double{a}{b}=\sum_i a_i \otimes b_i \in A\otimes A, \quad \widehat v (\alpha\otimes \beta)=\sum_j \alpha_j \otimes \beta_j \in M\otimes M,$$
  we have
\begin{equation*}\label{veryfirst} \bracket{a_\alpha}{b_\beta}_v =\sum_{i,j} \,\,  (a_i)_{\alpha_j}   (b_i)_{\beta_j} \in A_M.\end{equation*}
\end{theor}

Theorem~\ref{intro}    derives a Poisson bracket in $A_M$ from a double Poisson bracket in $A$ and a cyclic bilinear form on $M$. For constructions of double brackets, see \cite{VdB, VdW, MT2, MT}. We 
 give here  two   constructions  of cyclic   bilinear  forms. One construction starts from   a $\kk$-valued conjugation-invariant function on a finite group $G$ and  produces    a cyclic    form on the coalgebra $(\kk[G])^*$ dual to the group algebra $\kk[G]$.
In particular, this derives cyclic    forms on $(\kk[G])^*$ from the traces of linear representations of $G$  of finite rank.
The second construction starts from a symmetric Frobenius algebra  and produces a cyclic    form on the dual coalgebra.
For example, the algebra  of  matrices $\Mat_N(\kk)$ with $N\geq 1$ is a symmetric Frobenius algebra with Frobenius pairing  being  the trace of the  product of  matrices. This determines a cyclic bilinear form $v$ on    $(\Mat_N(\kk))^*$. The     bracket  $\bracket{-}{-}_v$  in~$A_N$  is the original Van den Bergh's    Poisson bracket. Note  that symmetric Frobenius algebras
naturally arise  in   Topological Quantum Field Theory, see   \cite{Ko}.

   We establish several properties of the   bracket $\bracket{-}{-}_v$  in $A_M$. If  the coalgebra  $M$ is counital, then    the  group $U(M^*)$ of invertible elements   of the   algebra $M^*$  naturally acts on $A_M$. The commutator $[\varphi, \psi]=  \varphi  \psi-   \psi\varphi$ turns $M^*$ into  a Lie algebra, $\underline M^* $, which also acts on $A_M$. The  actions  of $U(M^*)$ and    $\underline M^* $ are compatible in an appropriate sense and preserve the bracket  $\bracket{-}{-}_v$. For non-degenerate $v$,  we define  a natural trace map $A\to A_M  $ and compute  $\bracket{-}{-}_v$ on its   image.

When the algebra $A$ is unital, we define a quotient $A^+_M$ of $A_M$  characterized by the   universal  property as above with
$B$ running over unital commutative algebras. The bracket $\bracket{-}{-}_v$ in $A_M$   induces a bracket $\bracket{-}{-}_v^+$ in $A^+_M$, and the actions of $U(M^*)$ and  $\underline M^* $ on $A_M$   descend to $A^+_M$. We use the action of $\underline M^* $ to   formulate    Hamiltonian reduction  in this setting and  to study     double quasi-Poisson algebras.

We also introduce an  equivariant generalization of the bracket  $\bracket{-}{-}_v$ and  discuss connections with the topology of   surfaces.

The main body of the paper is written  in the more general setting of
graded algebras. Instead of   (double) Poisson brackets we consider their graded versions, the   (double) Gerstenhaber brackets.

 This work   was partially supported by the NSF
  grant  DMS-1202335.

\section{Representation algebras}\label{sec-AMT0}

   We begin with generalities on graded   algebras  and coalgebras and then introduce   the representation algebras.

\subsection{Graded  algebras}\label{gral}

By a {\it   module} we mean a   module over the fixed commutative ring $\kk$. By a {\it graded  module} we mean a $\ZZ$-graded  module $A=\oplus_{p\in \ZZ} \, A^p$.
An element $a$ of $A$ is {\it homogeneous} if    $a\in A^p$ for some $p$; we
 call $p$ the {\it degree} of $a$ and write   $\vert a\vert =p$.
 For any integer $n$, the  {\it $n$-degree} $\vert a\vert_n$ of a homogeneous element $a \in A$   is defined by  $\vert a\vert_n=  \vert a\vert+n$. Note that the $n$-degree   of $0\in A$ is an arbitrary integer.

An {\it   algebra} is  a     module   endowed with  an  associative
bilinear multiplication.
A {\it graded  algebra} is    a   graded module $A $ endowed with  an  associative
bilinear multiplication such that $ A^p A^q\subset A^{p+q}$ for all   $p,q\in \ZZ$.     Thus,
$\vert a b \vert =\vert a \vert +\vert b\vert$ for   homogeneous $a,b\in A$.
A graded algebra    $A$ is \emph{unital} if it has a  two-sided unit $1_A\in A^0$.  Unless explicitly stated to the contrary, we   do not require   algebras to be unital.

 For  a graded  algebra $A$,   denote by $[A,A]$ the  submodule of $A$ spanned by
the vectors $ab-(-1)^{\vert a \vert \vert b\vert} ba$ where $a,b$ run over all homogeneous
elements of $ A$.  The graded
algebra $A$  is {\it commutative} if   $[A,A]=0$.
Factoring an arbitrary graded algebra
 $A$ by the 2-sided ideal generated by $[A,A]$ we obtain a
commutative graded algebra  $\Com(A)$.

Given graded algebras $A$ and $B$,   a   {\it  graded  algebra homomorphism}
$A\to B$ is an  algebra homomorphism  $A \to  B$ carrying $A^p$ to $B^p$ for all $p\in \ZZ$. The   graded algebras  and graded algebra homomorphisms form a  category denoted $\grAlg$.

 Non-graded algebras will be identified with      graded algebras concentrated in degree $0$.


\subsection{Coalgebras}\label{MMM---} A {\it coalgebra} is a    module $M$ endowed with a coassociative linear map  ${\mu}:M\to M\otimes M$ called the {\it comultiplication} (here and below $\otimes=\otimes_{\kk}$). We do not suppose $M$ to be graded.
The image of any $\alpha\in M$ under ${{\mu}}$ expands (non-uniquely) as
  a   sum $  \sum_i \alpha^1_i \otimes
 \alpha^2_i$ where $ \alpha^1_i,
 \alpha^2_i \in  M$   and the
index $i$ runs over
a finite set. To shorten the formulas, we will   drop  the index $i$ and the summation sign
and write simply ${{\mu}}(\alpha)=   \alpha^1  \otimes
 \alpha^2$.    The coassociativity of ${{\mu}}$ reads then $(\alpha^1)^1\otimes (\alpha^1)^2 \otimes \alpha^2=\alpha^1 \otimes (\alpha^2)^1 \otimes (\alpha^2)^2$.

 A {\it counit} of a coalgebra
$M$ is     a linear map $\varepsilon =\varepsilon_M :M\to \kk$  such that   $\varepsilon (\alpha^1) \alpha^2=\varepsilon (\alpha^2) \alpha^1=\alpha$ for all $\alpha\in M$. If a counit exists,  then it is unique. A coalgebra having a counit is   \emph{counital}.
Unless explicitly stated to the contrary, we   do not require   coalgebras to be counital.

\subsection{The algebra  $\widetilde A_M$}\label{AMT0ee} From now on,  the symbol $A$ denotes a graded algebra and the symbol $M$ denotes a   coalgebra   with comultiplication ${{\mu}} $.
 Consider the algebra ${\widetilde A}_M$  with generators $\{a_{\alpha}\}$, where $a$ runs over   $A$ and $\alpha
$ runs over   $ M$, subject to the following two sets of relations:

 (i) (the bilinearity relations) for all  $k\in \kk$,  $a,b\in A$,  and $\alpha, \beta\in {M}$, \begin{equation*}\label{eq:addid} k \, a_{\alpha}=(ka)_{\alpha}=  a_{k\alpha} , \quad (a+b)_{\alpha}=  a_{\alpha} + b_{\alpha}, \quad
 a_{\alpha+\beta}= a_{\alpha} +a_{\beta}; \end{equation*}

 (ii) (the multiplicativity relations) for all   $a,b\in A$  and $\alpha \in {M}$,
\begin{equation*}\label{mult} (ab)_{\alpha}=    a_{\alpha^1} b_{\alpha^2} .
\end{equation*}

A fully developed  version of the latter relation is
  $(ab)_{\alpha}=  \sum_i  a_{\alpha^1_i} b_{\alpha^2_i}$ for an  expansion ${{\mu}}(\alpha)=   \sum_i \alpha^1_i \otimes
 \alpha^2_i$ with $\alpha^1_i , \alpha^2_i \in M$. The   bilinearity relations ensure that   $a_{\alpha^1} b_{\alpha^2}=\sum_i  a_{\alpha^1_i} b_{\alpha^2_i}$ does not depend on the choice of the expansion.

We turn ${\widetilde A}_M$ into a graded algebra by declaring that the generator    $  a_{\alpha}$ is homogeneous of  degree $ \vert a \vert  $ for all  homogeneous $a\in A $ and  all $\alpha\in   M $.  A typical element of ${\widetilde A}_M$ is represented by a non-commutative polynomial in the generators with zero free term.
One can alternatively define  ${\widetilde A}_{M}$ as the tensor algebra  $\oplus_{n\geq 1} (A\otimes M)^{\otimes n}$ quotiented by the relations $ ab \otimes \alpha =(a\otimes \alpha^1) (b\otimes \alpha^2) $ for all   $a,b\in A$, $\alpha\in M$.

The construction of $\widetilde A_{M}$ is functorial: a  graded    algebra homomorphism $f:A\to B$ induces a
 graded  algebra homomorphism $\widetilde f_M:\widetilde A_{M}\to \widetilde B_M$ by
$ \widetilde f_M(a_{\alpha})  = (f(a))_{\alpha}$ for all $a\in A$, $\alpha \in {M}$. For a fixed $M$, the formulas $  A\mapsto \widetilde A_{M}$, $f \mapsto \widetilde f_{M}$ define  an endofunctor of    the category of graded algebras  $\grAlg$.

\subsection{The universal  property}\label{MMM---convo} We now formulate the universal  property of $ \widetilde A_{M}$.
 Note that the  linear maps from   $M$  to  a  graded algebra $B $ form a graded module  $$H_M(B)=\Hom_\kk(M,B) =\oplus_{p } \,  \Hom_\kk(M,B^p).$$  The product of
$f_1, f_2\in H_M(B)$ is the   map $ {m_B} (f_1\otimes f_2){{\mu}}:M\to B$ where ${m_B}: B\otimes B\to B$ is the multiplication in~$B$. This makes $H_M(B)$ into  a graded algebra called the {\it convolution algebra}.
 The  map $  B\mapsto H_M(B) $  obviously extends to a functor
$H_M: \grAlg\to \grAlg$.


\begin{lemma}\label{just} Let $A$ be a graded algebra and $M$ be a coalgebra.
For any graded algebra $B$, there is a canonical bijection
\begin{equation}\label{eq:adjunction}
\Hom_{ \grAlg } ( \widetilde A_{M} , B)
\stackrel{\simeq}{\longrightarrow} \Hom_{\grAlg} (A,  H_M(B))
\end{equation}
which is natural in $A$ and $B$.
\end{lemma}

\begin{proof}
The map \eqref{eq:adjunction}  carries
a   graded algebra homomorphism $r:\widetilde A_{M}\to B$    to the
linear  map   $ s=s_r: A \to H_M(B)$  defined by
  $s(a)(\alpha)= r( a_{\alpha }) $ for all $a\in A$ and $\alpha\in M$.
  The map $s$ is grading-preserving and multiplicative: for  $a,b\in A$,
  $$s(a) s(b)(\alpha)=   {m_B} (s(a) \otimes s(b)){{\mu}} (\alpha)= {m_B} (s(a) \otimes s(b)) (\alpha^1 \otimes \alpha^2) $$
  $$= {m_B} (s(a) (\alpha^1)  \otimes s(b) (\alpha^2)) = r( a_{\alpha^1 }) r( b_{\alpha^2 })=r( a_{\alpha^1 }  b_{\alpha^2 }) =r((ab)_\alpha)= s(a b)(\alpha).$$

The map inverse to \eqref{eq:adjunction}  carries
     a graded algebra homomorphism  $s:A \to H_M(B)$ to the
algebra homomorphism   $ r=r_s: \widetilde A_{M}\to B$ defined on the   generators    by $r( a_{\alpha })=s(a)(\alpha)$. That   this rule  is compatible with   the bilinearity relations between the generators of $\widetilde A_M$ follows from the   linearity of $s$. The compatibility with the multiplicativity relations:
$$r(  (ab)_{\alpha})  = s(ab)(\alpha)= s(a) s(b) (\alpha)= {m_B} (s(a) \otimes s(b))  {{\mu}} (\alpha) $$
$$=  {m_B} (s(a) \otimes s(b))   (\alpha^1 \otimes \alpha^2) =s(a)   (\alpha^1 ) \, s(b)   (\alpha^2 ) = r( a_{\alpha^1 }) \, r( b_{\alpha^2 })= r( a_{\alpha^1 } b_{\alpha^2 } ).$$
 Clearly, $r$ preserves the grading,  and   the maps $s\mapsto r_s$ and $r\mapsto s_r$ are mutually inverse.
 \end{proof}

  Lemma~\ref{just} implies the following universal  property of $\widetilde A_M$: for any graded  algebra $B$ and any graded  algebra   homomorphism  $s:A\to H_M(B)$, there is a unique graded  algebra  homomorphism  $r:\widetilde  A_M\to B$ such that the   diagram
  $$
\xymatrix{ A  \ar[r]   \ar[rd]_{s} & H_M(\widetilde  A_M)\ar[d]^{H_M(r)}  \\
&   H_M(B)
}
$$ commutes.   Here the horizontal arrow is the  graded  algebra   homomorphism carrying   $a\in A$ to the   map $M\to \widetilde  A_M, \alpha\mapsto a_\alpha$. This homomorphism corresponds to the identity automorphism of $\widetilde A_M$ via \eqref{eq:adjunction}.

\subsection{The algebra $A_{M}$}\label{commAN} The commutative graded algebra $A_{M}=\Com(\widetilde A_{M}) $ is called the {\it representation algebra of $A$ with respect to $M$}. This algebra   is defined by the same generators and relations as
 $\widetilde A_{M}$ with additional  relations
$a_{\alpha} b_{\beta}=(-1)^{\vert a \vert
\vert b \vert} b_{\beta} a_\alpha$ for all homogeneous $a,b\in A$ and
$\alpha, \beta \in {M}$. The construction of $  A_{M}$ is
functorial: a morphism $f:A\to B$ in $\grAlg$ induces a morphism
$\widetilde f_M:\widetilde A_{M}\to \widetilde B_M$ in $\grAlg$, which in its turn
induces a morphism $  f_M:  A_{M}\to   B_M$ in the category of
commutative graded algebras  $\grCom$.
For any commutative graded algebra $B$,
$$
 \Hom_{\grCom} ({  A}_M, B) {\simeq}  \Hom_{\grAlg} (\widetilde A_{M}, B)
{\simeq} \Hom_{\grAlg} (A, H_M(B)).
$$
   The algebra $A_M$ has the same universal  property as $\widetilde A_M$  with  $B$ running  over commutative graded algebras.

 \subsection{Examples}\label{newexa} 1. Consider the coalgebra $M=(\Mat_N(\kk) )^*$ dual to the matrix algebra $\Mat_N(\kk)$ with $N\geq 1$. For    $i,j\in \{1,..., N\}$, let $\tau_{i,j} \in M$ be the linear   functional on $\Mat_N(\kk) $ carrying each matrix to its $(i,j)$-term.   These functionals form a basis of   $M $.  The comultiplication  in $M $    carries 
 $  \tau_{i,j}$ to $\sum_{r}  \tau_{i,r}\otimes \tau_{r,j}$ for all $i,j$.    
Given a graded algebra $A$, we write $a_{i,j}$ for the generator $a_{(\tau_{i,j})}  $ of   $\widetilde A_M$.  The defining relations of $\widetilde A_M$ in these generators are as follows:
$$(ka)_{i,j}=k\, a_{i,j}, \quad (a+b)_{i,j}=a_{i,j}+b_{i,j}, \quad (ab)_{i,j}=\sum_{r=1}^N a_{i,r} b_{r,j} $$
for all $a,b\in A$, $k\in \kk$, $i,j\in \{1,..., N\}$.

2.   Consider  the  group  algebra  ${\mathcal A}=\kk[G]$ of  a finite group  $G$ and the dual coalgebra $M={\mathcal A}^*$. The underlying module of $M$ is free with basis $\{\delta_g\}_{g\in G }$ dual to the basis $\{g\}_{g\in G}$ of ${\mathcal A}$.  The comultiplication in   $M$  is given by  $${{\mu}}(\delta_g)= \sum_{\substack{x,y \in G \\ xy =g}} \delta_{x} \otimes \delta_{y} \quad {\rm {for \,\, any}} \quad g\in G.$$  Given a graded algebra $A$, we write $a_{g}$ for the generator $a_{(\delta_g)}  $ of   $\widetilde A_M$.  The defining relations of $\widetilde A_M$ in these generators are as follows:
\begin{equation}\label{newexaeq}(ka)_{g}=k\, a_{g}, \quad (a+b)_{g}=a_{g}+b_{g}, \quad (ab)_{g}=\sum_{\substack{x,y \in G \\ xy =g}}   a_{x} b_{y} \end{equation}
for all $a,b\in A$, $k\in \kk$, $g\in G$.

3. If, in the previous example, $A=\kk[\pi]$ is the (non-graded) group  ring of a group $\pi$, then 
the algebra $\widetilde A_M$ is generated by the set $\{a_g\}_{a\in \pi, g\in G}$ subject only to the third relation in \eqref{newexaeq} for all $a,b\in \pi$, $g\in G$.


%

\section{Brackets and double brackets}\label{bibrackets}

We recall    the notions  of   Gerstenhaber  algebras and       double brackets.

 \subsection{Gerstenhaber  algebras}\label{gral++++---}   A   {\it    bracket}  in a graded module $B$  is a    linear map
${B}\otimes {B}\to {B}$. A bracket $\bracket{-}{-} $ in $B$ is {\it $n$-graded} for $n\in \ZZ$ if     $\bracket{{B}^p}{{B}^q}\subset {B}^{p+q+n}$ for all $p,q\in \ZZ$. An $n$-graded   bracket
 $\bracket{-}{-}$ in $ {B}$ determines a linear map $\{-,-,-\}: {B} \otimes  B \otimes  B  \to {B}$, called the \emph{Jacobi form},  by
 \begin{equation*}\label{jaco---} \{a,b,c\}= (-1)^{ \vert a \vert_n \vert c \vert_n} \bracket {a}  {\bracket{b}{c}}
+ (-1)^{ \vert b \vert_n\vert a \vert_n} \bracket  {b} {\bracket{c}{a}}
+(-1)^{  \vert c \vert_n \vert b \vert_n  }   \bracket   {c} {\bracket{a}{b}}
\end{equation*}
  for any homogeneous $a,b,c \in {B}$. Note   that $ \{a,b,c\}= \{ b,c,a\}= \{c, a,b \}$.

By an $n$-graded {\it biderivation} in  a  graded algebra ${A}$ we mean an $n$-graded  bracket $\bracket{-}{-} $ in $A$ such that 
for any homogeneous $a,b,c  \in {A}$,
\begin{equation}\label{poisson1} \bracket{a}{bc} = \bracket{a}{b}c +
(-1)^{\vert a\vert_n \vert b\vert} b \bracket{a}{c},
\end{equation}
\begin{equation}\label{poisson2} \bracket{ab}{c} =  a  \bracket{b}{c}   + (-1)^{\vert
b\vert \vert c\vert_n}     \bracket {a}{c} b, \end{equation}
 \begin{equation}\label{antis} \bracket{a}{b}=- (-1)^{ \vert a \vert_n \vert
b\vert_n} \bracket{b}{a}.
\end{equation}
   Formulas \eqref{poisson1} and \eqref{poisson2}  are the   {\it $n$-graded Leibniz rules}. 
   Formula \eqref{antis} is the   {\it $n$-graded antisymmetry}. Clearly, each of the Leibniz rules follows from the other one and the antisymmetry.

 An $n$-graded  biderivation    in a  graded algebra  
is    {\it  Gerstenhaber} if   its Jacobi form is equal to zero (this    is the {\it  $n$-graded Jacobi identity}).
An $n$-graded  Gerstenhaber biderivation     is also  called  an \emph{$n$-graded Gerstenhaber bracket}. A graded algebra equipped with  an $n$-graded Gerstenhaber bracket is called an \emph{$n$-graded Gerstenhaber algebra}.
  For non-graded algebras (i.e., for  graded algebras concentrated in degree zero) and $n=0$, we say \lq\lq Poisson" instead of  \lq\lq Gerstenhaber".

An example of a $0$-graded
Gerstenhaber  algebra   in provided by an arbitrary graded algebra with bracket    $ \bracket{a}{b} =    ab-(-1)^{\vert a\vert \vert
b\vert}  ba$ for    homogeneous  $a,b $.


\subsection{Double brackets}\label{bibr} Double brackets in   algebras were   introduced by Van den Bergh     \cite{VdB}. 
 The exposition here follows     \cite{MT}. 
We begin with notation.


  The   tensor product of $m\geq 2$ copies of a graded algebra $A$ is denoted  $A^{\otimes m}$.
For a permutation $(i_1,\dots,i_m)$ of $(1,\dots,m)$ and $n\in \ZZ$,
the symbol $P_{i_1\cdots i_m,n}$ denotes the  \emph{$n$-graded permutation} $A^{\otimes m}\to A^{\otimes m}$
carrying   $a_1\otimes \cdots \otimes a_m $   with homogeneous $a_1,\dots,a_m\in A$ to
$(-1)^{t} a_{i_1} \otimes   \cdots \otimes a_{i_m}$
 where $t\in \ZZ$ is the sum of the products $\vert a_{i_k} \vert_n \, \vert a_{i_l} \vert_n$
over all pairs of indices  $ k < l $  such that $i_k>i_l$.
 Set
$P_{i_1\cdots i_m } = P_{i_1\cdots i_m,0}$.

We define
two   $A$-bimodule structures on $A^{\otimes 2} $    by  $a(x\otimes y) b= ax \otimes  y b $
and
$$
a*(x\otimes y)*b= (-1)^{\vert a\vert \vert    xb\vert +\vert b\vert \vert y\vert} \, x b \otimes a y \quad
$$
for any homogeneous $a,b,x,y \in A$.

An    \emph{ $n$-graded  double bracket} in    $A$ with $n\in \ZZ$ is a  linear map $\double{-}{-}\in \End (A^{\otimes 2})$      satisfying the following conditions:

(i) (the grading condition) for
any integers $p,q$, $$
\double{A^p}{A^q}  \subset \bigoplus_{{i+j=p+q+n}}\, A^i \otimes A^j;
$$

(ii) (the $n$-graded antisymmetry) $\double{-}{-} P_{21,n}=- P_{21} \double{-}{-}$;

(iii) (the $n$-graded Leibniz rules) for all   homogeneous $a,b, c \in A$,
\begin{equation}\label{bibi}
\double{a}{bc} = \double{a}{b}c + (-1)^{\vert a\vert_n \vert b\vert} b \double{a}{c},
\end{equation}
\begin{equation}\label{bibi+}
\double{ab}{c} =  a* \double{b}{c}   + (-1)^{\vert
b\vert \vert c\vert_n}   \double{a}{c}*b.
\end{equation}

Note that   \eqref{bibi+} is a consequence of
\eqref{bibi} and the antisymmetry. Observe also
 that each   $x \in A\otimes A $ expands  (non-uniquely)
  as a   sum $ x=\sum_i x'_i \otimes
x''_i$ where $ x'_i  ,
x''_i\in A$ are homogeneous   and the
index $i$ runs over
a finite set. In the sequel, we will   drop   the index and the summation sign
and write simply $x =x'\otimes x''$. In this notation,
$$\double{a}{b}= \double{a}{b}'\otimes \double{a}{b}''$$
for any $a,b\in A$.   The $n$-graded antisymmetry condition  may be rewritten as
\begin{equation}\label{flip}\double{b}{a}=  - (-1)^{\vert a \vert_n \vert b \vert_n+\vert \double{a}{b}' \vert \vert \double{a}{b}'' \vert} \double{a}{b}''\otimes \double{a}{b}' \end{equation}
for any    homogeneous $a,b \in A$.

  An $n$-graded   double bracket  $\double{-}{-}$ in $A$ determines a linear endomorphism  $\triple{-}{-}{-} $ of $A^{\otimes 3}$, called the {\it induced    tribracket},  by
\begin{equation}\label{tribracket}
\triple{-}{-}{-} =
\sum_{i=0}^2  P_{312}^i ( \double{-}{-} \otimes \id_A) (\id_A \otimes \double{-}{-}) P_{312,n}^{-i}.
\end{equation}
 The   double bracket  $\double{-}{-}$
is   an {\it   $n$-graded  double Gerstenhaber  bracket} if the induced tribracket   is equal to zero.
The pair $(A,\double{-}{-})$ is called then an \emph{$n$-graded double Gerstenhaber algebra}.
 For non-graded algebras  and $n=0$, we say \lq\lq Poisson" instead of \lq\lq Gerstenhaber".

\section{Cyclic bilinear forms  and induced brackets}\label{bibrJac}

We define cyclic bilinear forms on coalgebras and  introduce the associated brackets in representation algebras.

 \subsection{The map ${\widehat v}$}\label{mapfv} The comultiplication ${{\mu}}$ in a coalgebra $M$ induces a linear map ${{\mu}}^m:M\to M^{\otimes m}$ for all $m\geq 2$. Namely, ${{\mu}}^2={{\mu}}$ and inductively, for $m\geq 3$,
  $${{\mu}}^m=(\id_{M^{\otimes k}} \otimes {{\mu}}  \otimes \id_{M^{\otimes (m-k-2)}}) {{\mu}}^{m-1}:M\to M^{\otimes m}$$
  where $k$ is any integer between $0$ and $m-2$. The coassociativity of ${{\mu}}$ ensures that ${{\mu}}^m$ does not depend on the choice of $k$.
For   $\alpha\in M$, we   write  $  {{\mu}}^m (\alpha)=   \alpha^1  \otimes
 \alpha^2    \otimes \cdots \otimes
 \alpha^m $
 where $\alpha^1, \alpha^2,..., \alpha^m \in M$ depend on an  index running over a finite set, and summation over this index is tacitly assumed.

 We    identify   elements of   $(M  \otimes M)^*=\Hom_\kk (M^{\otimes 2}, \kk)$ with $\kk$-valued bilinear forms on $M $.
 Each    $v\in (M  \otimes M)^*  $  determines  a linear map ${\widehat v}: M^{\otimes 2} \to M^{\otimes 2}$ by
  \begin{equation}\label{fv} {\widehat v}(\alpha\otimes \beta)=   v (\alpha \otimes \beta^2) \beta^1 \otimes \beta^3  \end{equation}
  for any $\alpha, \beta\in M$.

For a permutation $(i_1,\dots,i_m)$ of $(1,\dots,m)$ with $m\geq 2$,
the symbol $p_{i_1\cdots i_m }$ denotes the  linear map  $M^{\otimes m}\to M^{\otimes m}$
carrying any vector  $\alpha_1\otimes \cdots \otimes \alpha_m $    to
$ \alpha_{i_1} \otimes   \cdots \otimes \alpha_{i_m}$. In particular, $p_{21}: M^{\otimes 2}\to  M^{\otimes 2} $ is the   flip of the tensor factors.

 \begin{lemma}\label{propbiv}
 For any $v\in  (M  \otimes M)^* $,  the following two  diagrams   commute:
 $$
 \xymatrix@R=1cm @C=2cm {
 M \otimes M  \ar[d]_{\id_M\otimes {{\mu}}}   \ar[r]^-{{\widehat v}} &   {M \otimes M} \ar[d]^{\id_M\otimes {{\mu}}}\\
 M \otimes M \otimes M  \ar[r]^-{{\widehat v} \otimes \id_M}
&   {M \otimes M \otimes M},
}
$$
$$
\xymatrix@R=1cm @C=1.5cm{
M \otimes M    \ar[d]_{\id_M\otimes {{\mu}}}  \ar[r]^-{{\widehat v}} &   {M \otimes M}  \ar[rd]^{{{\mu}} \otimes \id_M}\\
M \otimes M \otimes M   \ar[r]^-{{p_{21}\otimes \id_M}}
&   {M \otimes M \otimes M}  \ar[r]^-{\id_M \otimes {\widehat v}} &  {M \otimes M \otimes M}
 .
}
$$
\end{lemma}

\begin{proof}   The commutativity of the first diagram: for any $\alpha, \beta \in M$,
$$({\widehat v}\otimes \id_M) (\id_M \otimes {{\mu}})(\alpha \otimes \beta )=  ({\widehat v}\otimes \id_M)(\alpha \otimes \beta^1\otimes   \beta^2)$$
$$=  v (\alpha \otimes \beta^2) \beta^1 \otimes \beta^3  \otimes   \beta^4 =(\id_M \otimes {{\mu}})(v(\alpha \otimes \beta^2) \beta^1 \otimes \beta^3)= (\id_M \otimes {{\mu}}) {\widehat v}(\alpha \otimes \beta ).$$
The commutativity of the second diagram is verified similarly: for  $\alpha, \beta \in M$,
$$(\id_M \otimes {\widehat v})(p_{21}\otimes \id_M) (\id_M \otimes {{\mu}}) (\alpha \otimes \beta )=
(\id_M \otimes {\widehat v}) (  \beta^1 \otimes \alpha \otimes \beta^2)$$
\begin{equation*}
=  v(\alpha \otimes \beta^3) \beta^1\otimes  \beta^2 \otimes \beta^4
=  v(\alpha \otimes \beta^2)   {{\mu}}  (  \beta^1) \otimes \beta^3  =({{\mu}} \otimes \id_M) {\widehat v} (\alpha \otimes \beta ). \qedhere\end{equation*}
\end{proof}

  Given $\alpha, \beta\in M$,  we can   expand  $  {\widehat v}(\alpha \otimes \beta)  $  as a  finite sum $\sum_i x_i\otimes y_i$ with $x_i, y_i\in M$. To shorten the formulas, we will suppress  the index  and the summation sign and write $\alpha_\beta$ for $x_i$ and $\beta^\alpha$ for $y_i$. Thus, we  write     ${\widehat v}(\alpha \otimes \beta)=\alpha_\beta \otimes \beta^\alpha$. In this notation, Lemma~\ref{propbiv} says that for any $\alpha, \beta \in M$,
\begin{equation}\label{ci23} {\alpha_\beta} \otimes { (\beta^\alpha)^1} \otimes { (\beta^\alpha)^2}={\alpha_{(\beta^1)}} \otimes {(\beta^1)^\alpha} \otimes {\beta^2},\end{equation}
\begin{equation}\label{ci24}  {(\alpha_\beta)^1}  \otimes  {(\alpha_\beta)^2}  \otimes  {\beta^\alpha}=
   {\beta^1}   \otimes {\alpha_{(\beta^2)}}  \otimes  {(\beta^2)^\alpha}.\end{equation}

\subsection{Cyclic bilinear forms}\label{cicrular}  We call  a bilinear form $v\in  (M  \otimes M)^*$ \emph{cyclic} if the endomorphisms ${\widehat v}$ and $p_{21}$ of $M^{\otimes 2}$ commute: ${\widehat v} p_{21}= p_{21}{\widehat v}$ or, equivalently, ${\widehat v} = p_{21}{\widehat v} p_{21} $. A bilinear form $v$ is cyclic if and only if for any $\alpha, \beta\in M$,
\begin{equation}\label{forcyc}     v(\alpha \otimes \beta^2) \beta^1 \otimes \beta^3=v(\beta \otimes \alpha^2) \alpha^3 \otimes \alpha^1 . \end{equation}
   In the notation above, this
can be rewritten     as
$   \alpha_\beta \otimes  \beta^\alpha = \alpha^\beta\otimes \beta_\alpha    $.

The set of cyclic bilinear forms on $M$ is a submodule of $(M  \otimes M)^*$.
 In particular,  $v=0$ is cyclic. More interesting examples    will be given in Sections~\ref{frfr----} and~\ref{FFFF}.

  The following lemma is our main technical result. It derives a bracket in the representation  algebra $ A_M$  from a   double bracket in   $A$ and a  cyclic    form on~$M$.

 \begin{lemma}\label{bra}
    Let  $\double{-}{-}$ be an $n$-graded  double bracket  in a graded algebra $A$ with $n\in \ZZ$. Let $v $ be a cyclic  bilinear form on a coalgebra $M$. There is a unique $n$-graded  biderivation   $ \bracket{-}{-}_v$ in $ A_M$   such that for any
  $a,b\in   A $  and $\alpha, \beta \in M$,
\begin{equation}\label{mainf}\bracket{a_\alpha}{b_\beta}_v= \double{a}{b}'_{\alpha_\beta} \double{a}{b}''_{\beta^\alpha} =v (\alpha \otimes \beta^2) \double{a}{b}'_{ \beta^1} \double{a}{b}''_{ \beta^3} \end{equation}
where we   expand  $\double{a}{b}=\double{a}{b}'\otimes \double{a}{b}''$
  as in Section~\ref{bibr}. The   associated Jacobi form $\{-,-,-\}_v$ in $ A_M$    is computed on the generators of $A_M$ as follows:   for  any homogeneous $a,b,c\in A$ and any $\alpha, \beta, \gamma\in M$,
  \begin{equation}\label{IMP1-}  \{a_{\alpha}, b_{\beta} , c_{\gamma}\}_v=Q-R\end{equation} where
\begin{equation}\label{IMP1} Q= Q(a,b,c, \alpha, \beta, \gamma)= (-1)^{\vert a
\vert_n \vert   c\vert_n } {\triple {a}  {b}{c}}_{\alpha_{(\beta_\gamma)}}'    {\triple {a}  {b}{c}}_{{(\beta_\gamma)^\alpha}}''
{\triple {a}  {b}{c}}_{{\gamma^\beta}}''',\end{equation}
\begin{equation}\label{IMP2}R= R(a,b,c, \alpha, \beta, \gamma)=  (-1)^{\vert   ab \vert
\vert   c \vert_n } {\triple {a}  {c}{b}}_{{\alpha_{(\gamma^\beta)}}}' {\triple {a}  {c}{b}}_{{{(\gamma^\beta)^\alpha}}}'' {\triple {a}  {c}{b}}_{{\beta_\gamma}}'''.
\end{equation}
\end{lemma}

\begin{proof}   The uniqueness of $ \bracket{-}{-}_v$ is obvious because the symbols ${a_\alpha}$ generate $ A_M$. To prove the existence,   consider the graded algebra $\A$ generated by the symbols $\{ a_\alpha\, \vert \, a\in A, \alpha \in M\}$ with $\vert a_\alpha\vert =\vert a\vert$ subject   to the bilinearity relations. The kernel of the natural projection $\A\to A_M$   is generated by the   multiplicativity
 relations  and the commutativity relations. It is clear that there is  a unique bilinear  pairing    $\bracket{-}{-}:\A\otimes \A\to A_M$  given by   \eqref{mainf} on the generators and satisfying the Leibniz rules. We claim that this  pairing  is antisymmetric and  the kernel of the projection $\A\to A_M$ lies in its   annihilator. This will imply the existence of $ \bracket{-}{-}_v$.

 Pick any homogeneous $a,b\in   A $ and set $x= \double{a}{b} \in A\otimes A$. For   $\alpha, \beta \in M$,
$$ \bracket{a_\alpha}{b_\beta}=x'_{\alpha_\beta} x''_{\beta^\alpha }=x'_{\alpha^\beta} x''_{\beta_\alpha } =(-1)^{ \vert x' \vert \vert x'' \vert} x''_{\beta_\alpha} x'_{\alpha^\beta}  $$
$$=-(-1)^{\vert a \vert_n \vert b \vert_n} \double{b}{a}'_{\beta_\alpha} \double{b}{a}''_{\alpha^\beta}=-(-1)^{\vert a \vert_n \vert b \vert_n} \bracket{b_\beta}{a_\alpha} $$
where the penultimate  equality follows from the antisymmetry of $\double{-}{-}$. Therefore, the   bracket $\bracket{-}{-}:\A\otimes \A\to A_M$ is antisymmetric in the sense that $\bracket{f}{g}=-(-1)^{\vert f\vert_n \vert g\vert_n} \bracket{g}{f}$ for any homogeneous $f,g\in \A$.

  Pick     homogeneous $ a,b, c\in   A $ and set $x= \double{a}{b}$, $y=\double{a}{c}$. Then
$$\double{a}{bc}=xc+ (-1)^{\vert a\vert_n \vert b\vert} by= x'\otimes x''c + (-1)^{\vert a\vert_n \vert b\vert} by'\otimes y''.$$ For      $\alpha, \beta \in M$,
$$\bracket{a_\alpha}{(bc)_\beta}= \double{a}{bc}'_{\alpha_\beta} \double{a}{bc}''_{\beta^\alpha}=x'_{\alpha_\beta} (x''c)_{\beta^\alpha}
+(-1)^{\vert a\vert_n \vert b\vert} (by')_{\alpha_\beta } y''_{\beta^\alpha}$$
$$
=x'_{\alpha_\beta} x''_{ (\beta^\alpha)^1} c_{ (\beta^\alpha)^2}+ (-1)^{\vert a\vert_n \vert b\vert} b_{(\alpha_\beta)^1} y'_{(\alpha_\beta)^2} y''_{\beta^\alpha}.$$
On the other hand,
$$\bracket{a_\alpha}{ b_{\beta^1} c_{\beta^2}}= \bracket{a_\alpha}{ b_{\beta^1} } c_{\beta^2}+  (-1)^{\vert a  \vert_n \vert  b  \vert}  b_{\beta^1} \bracket{a_\alpha}{ c_{\beta^2}}$$
$$=x'_{\alpha_{(\beta^1)}} x''_{(\beta^1)^\alpha} c_{\beta^2}+  (-1)^{\vert a  \vert_n \vert  b  \vert}  b_{\beta^1} y'_{\alpha_{(\beta^2)}} y''_{(\beta^2)^\alpha}.$$
 Now, the identities \eqref{ci23} and \eqref{ci24} imply that $\bracket{a_\alpha}{(bc)_\beta}=\bracket{a_\alpha}{ b_{\beta^1} c_{\beta^2}}$.
Therefore, the  bracket $\bracket{-}{-}:\A\otimes \A\to A_M$ annihilates   the multiplicativity relations. It also annihilates   the commutativity relations:
  \begin{eqnarray*} \bracket{a_\alpha}{b_\beta  c_\gamma}  & {=}&  \bracket{a_\alpha}{b_\beta }  c_\gamma + (-1)^{\vert a\vert_n \vert b\vert} b_\beta \bracket{a_\alpha}{  c_\gamma}      \\
&=& (-1)^{\vert ab\vert_n  \vert c\vert }  c_\gamma  \bracket{a_\alpha}{b_\beta }   + (-1)^{\vert a\vert_n \vert b\vert+\vert ac\vert_n \vert b\vert } \bracket{a_\alpha}{ c_\gamma} b_\beta   \\
&=&  (-1)^{\vert c\vert  \vert b\vert } \bracket{a_\alpha}{ c_\gamma} b_\beta +(-1)^{\vert a   b \vert_n   \vert c\vert }  c_\gamma  \bracket{a_\alpha}{b_\beta }     \\
&=&  \{ {a_\alpha},   (-1)^{\vert b\vert  \vert c\vert } c_\gamma b_\beta\} .
 \end{eqnarray*}

It remains only to prove \eqref{IMP1-}. In the rest of the argument, we write   $\bracket{-}{-}$ for $\bracket{-}{-}_v$.  The identities $P_{312,n}^{-1}=P_{312,n}^{2}=P_{231,n} $  imply that
\begin{eqnarray*}
\triple{a}{b}{c}&=& \double{a}{\double{b}{c}'}\otimes \double{b}{c}''
+  (-1)^{\vert a \vert_n \vert bc\vert} P_{312} \left (\double{b}{\double{c}{a}'} \otimes \double{c}{a}''\right) \\
&&+ (-1)^{ \vert ab\vert \vert c \vert_n} P_{231} \left (\double{c}{\double{a}{b}'} \otimes \double{a}{b}''\right)\\
&=& \double{a}{\double{b}{c}'}'\otimes \double{a}{\double{b}{c}'}''\otimes \double{b}{c}''\\
&&+ (-1)^{\vert a \vert_n \vert bc\vert}
P_{312} \left (\double{b}{\double{c}{a}'}' \otimes \double{b}{\double{c}{a}'}''   \otimes \double{c}{a}''\right)\\
&& + (-1)^{ \vert ab\vert \vert c \vert_n} P_{231}  \left (\double{c}{\double{a}{b}'}'
\otimes \double{c}{\double{a}{b}'}''  \otimes \double{a}{b}''\right).
\end{eqnarray*}
Using the  commutativity of $A_{M}$, we deduce that for any $\varphi, \psi  , \rho\in M$,
\begin{eqnarray}
\label{abcfff} && {\triple {a}  {b}{c}}_{\varphi}'    {\triple {a}  {b}{c}}_{\psi  }'' {\triple {a}  {b}{c}}_{\rho}'''\\
\notag &=&\double{a}{\double{b}{c}'}'_{\varphi} \double{a}{\double{b}{c}'}''_{\psi  } \double{b}{c}''_{\rho}\\
\notag &&+ (-1)^{\vert a \vert_n \vert bc\vert}
\double{b}{\double{c}{a}'}'_{\psi  }  \double{b}{\double{c}{a}'}''_{\rho}  \double{c}{a}''_{\varphi}\\
\notag && + (-1)^{ \vert ab\vert \vert c \vert_n}  \double{c}{\double{a}{b}'}'_{\rho}
\double{c}{\double{a}{b}'}''_{\varphi}   \double{a}{b}''_{\psi  }.
\end{eqnarray}
Setting  $\varphi={\alpha_{(\beta_\gamma)}}$, $\psi={{(\beta_\gamma)^\alpha}}$, $\rho={{\gamma^\beta}}$ and multiplying by $(-1)^{\vert a
\vert_n \vert   c\vert_n }$, we obtain
$$ Q= (-1)^{\vert a
\vert_n \vert   c\vert_n } u_1 + (-1)^{\vert a \vert_n \vert b \vert_n}
u_2 + (-1)^{ \vert  b\vert_n \vert c \vert_n} u_3 $$
where
\begin{eqnarray}
\label{abcggg}
\notag u_1 &=&\double{a}{\double{b}{c}'}'_{\alpha_{(\beta_\gamma)}} \double{a}{\double{b}{c}'}''_{{(\beta_\gamma)^\alpha}} \double{b}{c}''_{{\gamma^\beta}}\\
\notag u_2 &=&
\double{b}{\double{c}{a}'}'_{{(\beta_\gamma)^\alpha}}  \double{b}{\double{c}{a}'}''_{{\gamma^\beta}}  \double{c}{a}''_{\alpha_{(\beta_\gamma)}}\\
\notag u_3 &=&   \double{c}{\double{a}{b}'}'_{{\gamma^\beta}}
\double{c}{\double{a}{b}'}''_{\alpha_{(\beta_\gamma)}}   \double{a}{b}''_{{(\beta_\gamma)^\alpha}}.
\end{eqnarray}
Permuting $b$ and $c$ in \eqref{abcfff}, we similarly obtain that
$$R=(-1)^{\vert   ab \vert
\vert   c \vert_n } t_1 + (-1)^{ \vert ac\vert \vert b \vert_n}
t_2 + (-1)^{  \vert  bc\vert \vert a  \vert_n} t_3$$
where
\begin{eqnarray}
\notag t_1 &=&\double{a}{\double{c}{b}'}'_{{\alpha_{(\gamma^\beta)}}} \double{a}{\double{c}{b}'}''_{{{(\gamma^\beta)^\alpha}}} \double{c}{b}''_{{\beta_\gamma}}\\
\notag t_2 &=&
\double{c}{\double{b}{a}'}'_{{{(\gamma^\beta)^\alpha}}}  \double{c}{\double{b}{a}'}''_{{\beta_\gamma}}  \double{b}{a}''_{{\alpha_{(\gamma^\beta)}}}\\
\notag t_3 &=&  \double{b}{\double{a}{c}'}'_{{\beta_\gamma}}
\double{b}{\double{a}{c}'}''_{{\alpha_{(\gamma^\beta)}}}   \double{a}{c}''_{{{(\gamma^\beta)^\alpha}}}.
\end{eqnarray}

We next compute $ \{a_{\alpha}, b_{\beta} , c_{\gamma}\}_v$.
Set $x=\double {b} {c}\in A^{\otimes 2}$ and observe that
$$
\{ {a_{\alpha}}, {\bracket {b_{\beta}} {c_{\gamma}}} \}
=\{ {a_{\alpha}},{x'_{{\beta_\gamma}} x''_{{\gamma^\beta}}} \}= \{ {a_{\alpha}}, {x'_{{\beta_\gamma}} } {x}''_{{\gamma^\beta}}\} + (-1)^{\vert a\vert_n  \vert x'\vert}  x'_{{\beta_\gamma}} \{ {a_{\alpha}} , {{x}''_{{\gamma^\beta}}}\}$$
$$
= \double {a}  {x'}'_{\alpha_{(\beta_\gamma)}}  \double{a}{ x'}''_{{(\beta_\gamma)^\alpha}} {x}''_{{\gamma^\beta}}
+ (-1)^{\vert a\vert_n  \vert x'\vert}  x'_{{\beta_\gamma}} \double {a}  {x''}'_{{\alpha_{(\gamma^\beta)}}}  \double{a}{ x''}''_{{{(\gamma^\beta)^\alpha}}}.
$$
  We   rewrite  the last summand as follows.  Note that
$$\vert \double {a}  {x''}'_{{\alpha_{(\gamma^\beta)}}}
\double{a}{ x''}''_{{{(\gamma^\beta)^\alpha}}} \vert = \vert \double {a}  {x''}'
\double{a}{ x''}''  \vert=\vert a \vert +\vert x''\vert +n =\vert a \vert_n +\vert x''\vert  .$$
Also, by \eqref{flip},
$x'\otimes x'' =  -(-1)^{\vert b\vert_n \vert c\vert_n+\vert y'\vert \vert y''\vert} y''\otimes y'$ where $y=  \double{c}{b}$. These formulas and the   commutativity of $A_{M}$ imply that
$$(-1)^{\vert a\vert_n  \vert x'\vert}  x'_{{\beta_\gamma}} \double {a}  {x''}'_{{\alpha_{(\gamma^\beta)}}}
\double{a}{ x''}''_{{{(\gamma^\beta)^\alpha}}}$$
$$= (-1)^{   \vert x'\vert   \vert x''\vert  } \double {a}  {x''}'_{{\alpha_{(\gamma^\beta)}}}
\double{a}{ x''}''_{{{(\gamma^\beta)^\alpha}}} x'_{{\beta_\gamma}} $$
$$=  -(-1)^{\vert b\vert_n \vert c\vert_n}
\double {a}  {y'}'_{{\alpha_{(\gamma^\beta)}}}
\double{a}{ y'}''_{{{(\gamma^\beta)^\alpha}}} y''_{{\beta_\gamma}}. $$
As a result, we obtain  that
$$
(-1)^{\vert a\vert_n \vert c\vert_n} \bracket {a_{\alpha}} {\bracket {b_{\beta}} {c_{\gamma}}}
 =(-1)^{\vert a\vert_n \vert c\vert_n} w_1
  - (-1)^{\vert ab\vert  \vert c\vert_n} z_1  $$
where
\begin{eqnarray} \label{abc_simple2}
 \notag w_1 &=&  \double {a}  {\double{b}{c}'}'_{\alpha_{(\beta_\gamma)}}  \double{a}{ \double{b}{c}'}''_{{(\beta_\gamma)^\alpha}} {\double {b}
{c}}''_{{\gamma^\beta}} , \\
 \notag  z_1 &=&\double {a}  {\double {c} {b}'}'_{{\alpha_{(\gamma^\beta)}}}
\double{a}{ \double {c} {b}'}''_{{{(\gamma^\beta)^\alpha}}} \double {c} {b}''_{{\beta_\gamma}}
 .\end{eqnarray}
Cyclically permuting $a_{\alpha}, {b_{\beta}}, c_{\gamma}$, we obtain
\begin{eqnarray}
 \notag (-1)^{\vert a\vert_n \vert b\vert_n} \bracket {b_{\beta}}  {\bracket  {c_{\gamma}} {a_{\alpha}}}
&=& (-1)^{\vert a\vert_n \vert b\vert_n} w_2- (-1)^{ \vert bc\vert \vert a\vert_n } z_2,\\
 \notag (-1)^{\vert b\vert_n \vert c\vert_n}\bracket  {c_{\gamma}}  {\bracket  {a_{\alpha}}  {b_{\beta}}}
&=&(-1)^{\vert b\vert_n \vert c\vert_n} w_3- (-1)^{ \vert ac\vert \vert b\vert_n } z_3
  \end{eqnarray}
where
\begin{eqnarray}
\notag
w_2 &=&\double {b}  {\double {c} {a}'}'_{ \beta_{(\gamma_\alpha)}}  \double{b}{ \double {c} {a}'}''_{{(\gamma_\alpha)^\beta}} {\double {c}{a}}''_{\alpha^\gamma},\\
\notag z_2 &=&      \double {b}  {\double {a} {c}'}'_{ \beta_{(\alpha^\gamma)}}
\double{b}{ \double {a} {c}'}''_{(\alpha^\gamma)^\beta  } \double {a} {c}''_{ \gamma_\alpha },\\
\notag
w_3 &=&\double {c}  {\double {a} {b}'}'_{ \gamma_{(\alpha_\beta)}}  \double{c}{ \double {a} {b}'}''_{{(\alpha_\beta)^\gamma}} {\double {a}{b}}''_{\beta^\alpha},\\
\notag z_3 &=&      \double {c}  {\double {b} {a}'}'_{ \gamma_{(\beta^\alpha)}}
\double{c}{ \double {b} {a}'}''_{(\beta^\alpha)^\gamma  } \double {b} {a}''_{ \alpha_\beta }.
\end{eqnarray}
Below we prove that    \begin{equation*}\label{sdf} w_1=u_1,  \quad  w_2=u_2, \quad w_3=u_3, \quad z_1=t_1, \quad  z_2=t_3, \quad z_3=t_2. \end{equation*}
Therefore $$\{a_{\alpha}, b_{\beta} , c_{\gamma}\}_v=$$  $$(-1)^{\vert a\vert_n \vert c\vert_n} \bracket {a_{\alpha}} {\bracket {b_{\beta}} {c_{\gamma}}}+ (-1)^{\vert a\vert_n \vert b\vert_n} \bracket {b_{\beta}}  {\bracket  {c_{\gamma}} {a_{\alpha}}}
+(-1)^{\vert b\vert_n \vert c\vert_n}\bracket  {c_{\gamma}}  {\bracket  {a_{\alpha}}  {b_{\beta}}}$$
$$=
(-1)^{\vert a\vert_n \vert c\vert_n} u_1
  + (-1)^{\vert a\vert_n \vert b\vert_n} u_2  +(-1)^{\vert b\vert_n \vert c\vert_n} u_3 $$
  $$ - (-1)^{\vert ab\vert  \vert c\vert_n} t_1- (-1)^{ \vert bc\vert \vert a\vert_n } t_3 - (-1)^{ \vert ac\vert \vert b\vert_n } t_2=Q-R.$$

 The equalities $w_1=u_1$ and $z_1=t_1$ are tautological.   The formula $w_2=u_2$ would follow from the identity
  \begin{equation}\label{F--F}  \beta_{(\gamma_\alpha)}  \otimes (\gamma_\alpha)^\beta\otimes {\alpha^\gamma}=(\beta_\gamma)^\alpha \otimes \gamma^\beta\otimes \alpha_{(\beta_\gamma)}.\end{equation}
  Similarly, the formulas $w_3=u_3,  z_2=t_3,  z_3=t_2$
  would follow from the identities
 \begin{equation}\label{F--F+1}  \gamma_{(\alpha_\beta)} \otimes (\alpha_\beta)^\gamma \otimes \beta^\alpha=\gamma^\beta \otimes \alpha_{(\beta_\gamma)}\otimes (\beta_\gamma)^\alpha, \end{equation}
 \begin{equation}\label{F--F+2} { \beta_{(\alpha^\gamma)}} \otimes {(\alpha^\gamma)^\beta  } \otimes { \gamma_\alpha }=   {{\beta_\gamma}}
\otimes {{\alpha_{(\gamma^\beta)}}}   \otimes {{{(\gamma^\beta)^\alpha}}},\end{equation}
\begin{equation}\label{F--F+3} { \gamma_{(\beta^\alpha)}} \otimes {(\beta^\alpha)^\gamma  } \otimes { \alpha_\beta }= {{{(\gamma^\beta)^\alpha}}}  \otimes {{\beta_\gamma}}  \otimes {{\alpha_{(\gamma^\beta)}}}.\end{equation}
Formula \eqref{F--F+1} is deduced from \eqref{F--F} by replacing $\alpha, \beta, \gamma$ with $\gamma, \alpha, \beta$, respectively,  and permuting the tensor factors. Formula \eqref{F--F+2} is deduced from \eqref{F--F+3} by replacing $\alpha, \beta, \gamma$ with $ \beta, \gamma, \alpha$, respectively, and permuting the tensor factors. It remains to prove \eqref{F--F} and \eqref{F--F+3}.

  We rewrite \eqref{F--F} in the equivalent form
  \begin{equation}\label{F--} {\alpha^\gamma}\otimes  \beta_{(\gamma_\alpha)}  \otimes (\gamma_\alpha)^\beta=  \alpha_{(\beta_\gamma)} \otimes (\beta_\gamma)^\alpha \otimes \gamma^\beta .\end{equation}
  Set   $${X}=\widehat v \otimes \id_M \in \End M^{\otimes 3}  \quad {\text {and}} \quad {Y}=\id_M \otimes \widehat v \in \End M^{\otimes 3}.$$
The sides of \eqref{F--} are the images of $\beta \otimes \gamma\otimes \alpha\in M^{\otimes 3}$  under  $p_{312} {X} {Y}$ and ${X}{Y}p_{312}$. Thus,   \eqref{F--} follows from the   equality \begin{equation}\label{F11} p_{312} {X} {Y}= {X}{Y}p_{312}  \end{equation}
   which we now prove.  For   $\alpha, \beta, \gamma \in M$,
\begin{eqnarray*}
    XY(\alpha\otimes \beta\otimes \gamma)
  &=&  X (  v(\beta \otimes \gamma^2) \alpha\otimes \gamma^1 \otimes \gamma^3 ) \\
\notag &=& v(\alpha \otimes \gamma^2) v(\beta \otimes \gamma^4)   \gamma^1 \otimes \gamma^3 \otimes \gamma^5  .
\end{eqnarray*}
Hence,
$$p_{312} XY(\alpha\otimes \beta\otimes \gamma)=v(\alpha \otimes \gamma^2)
v(\beta \otimes \gamma^4)    \gamma^5 \otimes \gamma^1 \otimes \gamma^3.$$
Similarly, \begin{eqnarray*}
    XY p_{312} (\alpha\otimes \beta\otimes \gamma)
\notag &=&  XY (\gamma \otimes \alpha\otimes \beta) \\
\notag &=&   v( \alpha \otimes \beta^4 )v( \gamma \otimes \beta^2 )    \beta^1 \otimes  \beta^3 \otimes  \beta^5.
\end{eqnarray*}
We must  show that \begin{equation}\label{loc34}  v( \alpha \otimes \beta^4 )v( \gamma \otimes \beta^2 )    \beta^1 \otimes  \beta^3 \otimes  \beta^5= v(\alpha \otimes \gamma^2) v(\beta \otimes \gamma^4)   \gamma^5 \otimes \gamma^1 \otimes \gamma^3  . \end{equation}
To this end, we take \eqref{forcyc} and replace $\alpha $  with   $\gamma$. This gives
 $$v(\gamma \otimes \beta^2) \beta^1 \otimes \beta^3 =  v(\beta \otimes \gamma^2) \gamma^3 \otimes \gamma^1 .$$
 Applying $\id_M \otimes {{\mu}}^3$ to both sides, we obtain that
 $$v(\gamma \otimes \beta^2) \beta^1 \otimes \beta^3 \otimes \beta^4 \otimes \beta^5=  v(\beta \otimes \gamma^4) \gamma^5 \otimes \gamma^1 \otimes \gamma^2\otimes \gamma^3 .$$
 We apply to both sides the linear map $M^{\otimes 4}\to M^{\otimes 3}$  carrying any  $x\otimes y\otimes z \otimes t$ to $v(\alpha \otimes z)
 x\otimes y  \otimes t$. This gives \eqref{loc34} and completes the proof of  \eqref{F--F}.

To prove \eqref{F--F+3}, we rewrite  it in the following equivalent   form:
 \begin{equation*}  { \alpha_\beta } \otimes { \gamma_{(\beta^\alpha)}} \otimes {(\beta^\alpha)^\gamma  }= {{\alpha_{(\gamma^\beta)}}} \otimes {{{(\gamma^\beta)^\alpha}}}  \otimes {{\beta_\gamma}} .\end{equation*}
 The latter   can be reformulated in terms of the maps $X,Y$ as the equality \begin{equation}\label{F14}  {Y} p_{132} {X}= {X} p_{132} {Y}    . \end{equation}
  To prove \eqref{F14}, set $\sigma =p_{213}$  and $\tau=p_{132}$. Clearly $\sigma^2=\tau^2=1$, $\sigma \tau \sigma=\tau \sigma\tau$, $p_{312}=\sigma \tau$.  It follows from the definition of $X$ and $Y$ that ${Y}=p_{312} {X} p_{312}^{-1}=\sigma\tau {X} \tau\sigma$. Formula \eqref{F11}
   may be rewritten in this notation as
   $$ \sigma \tau {X} \sigma\tau {X} \tau\sigma={X} \sigma\tau {X} \tau\sigma \sigma\tau ={X} \sigma\tau {X}.$$
 The cyclicity of $v$ gives $\sigma {X}={X} \sigma$ and $\tau Y=Y\tau$.  Then
  $${Y} p_{132} {X} ={Y} \tau {X}  = \sigma \tau {X} \tau\sigma \tau {X} = \sigma \tau {X}  \sigma \tau \sigma {X}$$
  $$=\sigma \tau  {X} \sigma \tau {X} \sigma   =\sigma \tau  {X} \sigma \tau {X} \sigma \tau \sigma \sigma \tau= \sigma \tau  {X} \sigma \tau {X}  \tau \sigma \tau \sigma \tau$$
  \begin{equation*} ={X} \sigma\tau {X} \tau \sigma \tau={X} {Y} \tau=X\tau Y={X}  p_{132} {Y}. \qedhere \end{equation*}
\end{proof}

We can now formulate the main result  of this paper.

\begin{theor}\label{th-main} For any  $n$-graded   double Gerstenhaber bracket $\double{-}{-}$ in a graded algebra  $A$ and any   cyclic bilinear form $v$ on a coalgebra  $M$, the  $n$-graded  biderivation  $\bracket{-}{-}_v$ in the representation algebra $A_M$  is an $n$-graded Gerstenhaber  bracket.
\end{theor}

\begin{proof}   We need to prove that the  Jacobi form $\{-,-,-\}_v$ in $ A_M$ associated  with the bracket $\bracket{-}{-}_v$ is equal to zero.  Using the $n$-graded Leibniz rules and the   antisymmetry for   $ \bracket{-}{-}_v$, we easily compute   that   the  Jacobi form   satisfies the following Leibniz-type formula: for any $a_1,a_2,b,c\in A_M$, \begin{equation}\label{jjja} \{a_1a_2,b,c\}_v
=(-1)^{\vert a_1  \vert  \vert c \vert_n  } a_1\{a_2, b,c\}_v+   (-1)^{   \vert  a_2 \vert  \vert b  \vert_n} \{a_1, b,c\}_v a_2 .\end{equation}
 This   and the cyclic invariance of the Jacobi form imply
that if this form is zero on the generators of $A_M$, then it is zero on all elements of~$A_M$.  The  theorem now follows from Lemma~\ref{bra} and the assumption  $\triple{-}{-}{-}=0$.
\end{proof}

      For non-graded algebras, Theorem~\ref{th-main} yields Theorem~\ref{intro} of the introduction.

\section{Cyclic
 structures on    algebras and coalgebras}\label{frfr----}

We reformulate cyclic bilinear forms on coalgebras in terms of  so-called    cyclic structures on     coalgebras and  algebras.

\subsection{Cyclic  structures on coalgebras}  A {\it cyclic structure} on a coalgebra  $M$ is a linear map $M\to M^*, \alpha \mapsto \overline \alpha$ such that for any $\alpha, \beta \in M$, \begin{equation}\label{cycl1}   {\overline \alpha} (\beta^2) \beta^1 \otimes \beta^3 =   {\overline \beta} (\alpha^2)   \alpha^3 \otimes \alpha^1 .\end{equation}


   \begin{lemma}  For   $ v\in (M\otimes M)^*$,  we let  $\ad_v: M\to M^*$ be the   left adjoint map carrying any $\alpha\in M$ to the linear map $ M\to \kk, \beta \mapsto v (\alpha \otimes \beta)$. Then:

    (a)    $ v\in (M\otimes M)^*$ is cyclic if and only if   $\ad_v $   is a cyclic structure on $M$;

     (b) the formula $v\mapsto \ad_v$ establishes a bijection of the set of cyclic bilinear forms on $M$ onto the set of cyclic structures on $M$.
\end{lemma}

\begin{proof}
 For    $\alpha, \beta\in M$,
 $$ {\widehat v}p_{21} (\beta\otimes \alpha) = {\widehat v}(\alpha\otimes \beta)=  v( \alpha \otimes \beta^2) \beta^1 \otimes \beta^3 =  \ad_v (\alpha) (\beta^2) \beta^1 \otimes \beta^3 $$
 and
$$p_{21} {\widehat v}(\beta\otimes \alpha)=  p_{21} ( v (\beta \otimes \alpha^2) \alpha^1 \otimes \alpha^3) =  v (\beta \otimes \alpha^2)  \alpha^3 \otimes \alpha^1  =  {\ad_v( \beta)}(\alpha^2) \alpha^3 \otimes \alpha^1.$$
Clearly,  $ {\widehat v}p_{21} =p_{21} {\widehat v}$   if and only if $\ad_v$ is a cyclic structure on $M$.
  This proves (a).
Claim  (b) is obvious; the inverse bijection carries a cyclic structure  $M\to M^*, \alpha \mapsto \overline \alpha$ to the  bilinear form   $\alpha \otimes \beta \mapsto \overline \alpha (\beta)$ on $M$.
\end{proof}

\subsection{Cyclic  structures on algebras}\label{fralpp} Dualizing   cyclic structures on coalgebras we obtain the following notion.
A {\it cyclic structure} on  an algebra $\A$ is a linear map $\A^*\to \A, \alpha \mapsto \overline \alpha$ such that $\alpha (a{\overline \beta} b )= \beta (b{\overline \alpha} a ) $ for any $\alpha, \beta \in \A^* $ and $a,b \in \A$. Here and below $\A^*=\Hom_\kk({\mathcal A} ,\kk)$.


Given    a coalgebra $M$, a cyclic structure $ M^{**}\to M^{*} , x\mapsto \overline x$ on the dual algebra $M^*$ induces (under additional assumptions) a cyclic structure on $M$. Recall that   multiplication  in $M^*$ is defined by $(ab)(\alpha)=a(\alpha^1) b(\alpha^2)$ for       $a,b\in M^*$ and  any $\alpha\in M$. Let $e:M\to M^{**}$ be the evaluation map carrying   $\alpha\in M$ to the functional $M^*\to \kk, a\mapsto a(\alpha)$. For     $\alpha, \beta\in M$, $a,b\in M^*$,
$$a(\beta^1) \, \overline {e(\alpha)} (\beta^2)\,   b(\beta^3)=(a\,  \overline {e(\alpha)}\,  b) (\beta)=  e(\beta) (a\,  \overline {e(\alpha)}\,  b) $$
$$ = e(\alpha)( b \, \overline {e(\beta)} a)= ( b \, \overline {e(\beta)} a) (\alpha)= b(\alpha^1) \, \overline {e(\beta)} (\alpha^2)\,   a(\alpha^3).$$
In other words, the evaluations of $a\otimes b$ on the   vectors    \begin{equation}\label{vectros} \overline {e(\alpha)} (\beta^2) \, \beta^1 \otimes \beta^3, \quad \overline {e(\beta)} (\alpha^2)\, \alpha^3 \otimes \alpha^1 \in M\otimes M \end{equation} are equal. Since this holds for all $a,b\in M^*$, we can deduce the following: if the underlying module of $M$ is free, then the vectors \eqref{vectros} are equal for all $\alpha, \beta \in M$. This means that   $\overline e: M\to M^*$ is a cyclic structure on   $M$.

We will often focus on algebras and coalgebras whose underlying modules are free of finite rank.
Any such algebra $\A$ is dual to a well defined  coalgebra $( M= \A^*,{{\mu}})$ where   ${{\mu}}: M \to M\otimes M$ is   determined by the condition that  the evaluation of ${{\mu}}(\alpha)$ on $a\otimes b$ is equal to $\alpha(ab)$ for any $\alpha\in M$,  $a,b\in \mathcal A$.    The evaluation  map $ M\to M^{**}$ is an isomorphism and we use it to identify $ M^{**}$ with~$M$. The arguments above show  that
cyclic structures on $M$ and   $\A$ are the same maps  $M\to \A $.  The
 cyclic bilinear form  $ v: M\otimes M\to \kk$   associated with a cyclic structure $M\to \A , \alpha \mapsto \overline \alpha$   is  computed by $  v( \alpha \otimes \beta )=\beta( {\overline \alpha} )$ for $ \alpha, \beta \in M$.

\subsection{Example}\label{exaafral} In   Example \ref{newexa}.2, any conjugation-invariant function  $F:G\to \kk$ determines a cyclic structure  $M=\A^* \to \A$ on~$\A$  and on $M$   by $$\overline {\delta_g}=\sum_{z\in G} F(gz) z \in \A  \quad {\rm { for }} \quad g\in G .$$
The cyclic identity is easily verified on the basis vectors: for  any $a,b, g,h\in G$,
  \begin{eqnarray*} \delta_g (a \overline {\delta_h} b)  & {=}&  \delta_g (  \sum_{z\in G} F(h z) a z b)    = F(h a^{-1} g b^{-1})   \\
&=&  F(g b^{-1} h a^{-1})   = \delta_h (  \sum_{z\in G} F(g z) b z a) =\delta_h (b \overline {\delta_g}a).
 \end{eqnarray*}
  The associated
  cyclic bilinear form $ v  $ on $M$  is computed by $v (\delta_g\otimes \delta_h)=F(gh)$ for   $g,h \in G$.
   Given a graded algebra $A$ with double bracket $\double{-}{-}$,   the      bracket $\bracket{-}{-}_v$ in $A_M$    is computed from   \eqref{mainf}: for $a,b\in A$ and $g,h \in G $,
$$\bracket{a_{ g}}{b_{h}}_v=  \sum_{x,y\in G  } F(gx^{-1} h y^{-1}) \double{a}{b}'_{x} \double{a}{b}''_{y}.$$
 Note that a   linear representation $\rho:G\to  GL_N(\kk)$   with $N\geq 1$ determines  a conjugation-invariant function $   G\to \kk, g\mapsto  \tr (\rho(g))$. The corresponding cyclic bilinear form   on~$M$ carries $\delta_g\otimes \delta_h$ to $\tr(\rho(gh))$ for   any $g,h \in G$.

\section{Cyclic bilinear forms from Frobenius algebras}\label{FFFF}

We show how to derive cyclic bilinear forms   from symmetric Frobenius algebras. We begin by recalling basic definitions concerning Frobenius algebras.

\subsection{Frobenius algebras} A $\kk$-valued bilinear form is   {\it non-degenerate} if its left adjoint map is an isomorphism.
  A {\it Frobenius algebra} is a pair  consisting of  an   algebra ${{{{\mathcal A}}}}$ whose underlying module is free of finite rank   and   a non-degenerate bilinear form $ (-,-): {{{{{\mathcal A}}}}} \times {{{{{\mathcal A}}}}}\to \kk$   such that $ ( ab, c)= ( a, bc)$ for any $a,b,c\in {{{{\mathcal A}}}}$.  The algebra ${{{{\mathcal A}}}}$ is not required to be graded or unital. The form $ (-,-)$ is called  the {\it Frobenius pairing}.

A Frobenius algebra $ {{{{\mathcal A}}}} $ is {\it symmetric} if the Frobenius pairing is symmetric, i.e., $  ( a, b )=  ( b,a)$ for any $a,b\in {{{{{\mathcal A}}}}}$.
An example of a  symmetric Frobenius algebra is provided by the ring $\kk$ with the ring multiplication in the role of the Frobenius pairing.
 Any unital commutative  Frobenius algebra ${{{{\mathcal A}}}}$ is   symmetric: for $a,b\in {{{{\mathcal A}}}}$,
$$(a, b)=(a , b 1_{{{{{\mathcal A}}}}})=(ab, 1_{{{{{\mathcal A}}}}})=(ba, 1_{{{{{\mathcal A}}}}})=(b, a1_{{{{{\mathcal A}}}}})= (b,a).$$


\subsection{From Frobenius algebras to cyclic forms}\label{frofro}  The following theorem derives from every symmetric Frobenius algebra a cyclic bilinear form on the dual coalgebra.

  \begin{theor}\label{CQPS+}     Let $ {{{{\mathcal A}}}} $ be a symmetric Frobenius algebra with Frobenius pairing $(-,-)$.
      For any $\alpha\in {{{{{\mathcal A}}}}}^*$, there is a unique  ${\overline \alpha}\in {{{{\mathcal A}}}} $ such that $\alpha (a)=  ( a,  {\overline \alpha} )   $ for all $a\in {{{{\mathcal A}}}}$. The map ${{{{{\mathcal A}}}}}^*\to  {{{{{\mathcal A}}}}}, \alpha \mapsto \overline \alpha$ is a cyclic structure on ${{{{{\mathcal A}}}}}$ and on ${{{{{\mathcal A}}}}}^*$.
      The associated cyclic
   bilinear form $  v $ on   ${{{{{\mathcal A}}}}}^*  $   
  is given   by $v (\alpha \otimes \beta)=\beta( {\overline \alpha} )= (\overline \alpha, \overline \beta)$ for any   $ \alpha, \beta \in {{{{{\mathcal A}}}}}^*$.
\end{theor}

\begin{proof} The existence and uniqueness of ${\overline \alpha} $ follows from the non-degeneracy of the Frobenius pairing. For any $\alpha, \beta \in {{{{{\mathcal A}}}}}^* $ and $a,b \in  {{{{\mathcal A}}}}$,
\begin{equation*} \alpha (a{\overline \beta} b )=  ( a{\overline \beta} b, {\overline \alpha} )=  ( a{\overline \beta}, b  {\overline \alpha} )
=  (  b  {\overline \alpha}, a{\overline \beta} )=  (  b  {\overline \alpha}  a, {\overline \beta} )=\beta (b{\overline \alpha} a).\end{equation*}
The computation of $v$ follows from the remark  at the end of Section~\ref{fralpp}.
\end{proof}

We will  combine Theorem \ref{CQPS+}   with the following construction of   symmetric Frobenius algebras. Let ${{{{\mathcal A}}}}$ be an algebra whose underlying module is free of finite rank.
We say that $\theta \in {{{{\mathcal A}}}}^*$ is {\it trace-like} if the bilinear form $ (-,-)_\theta$ on ${{{{\mathcal A}}}}$ defined by $ (a,b)_\theta=\theta (ab)$ for $a,b\in {{{{\mathcal A}}}}$ is symmetric and non-degenerate. Then the pair $({{{{\mathcal A}}}}, (-,-)_\theta)$ is a symmetric Frobenius algebra.  The associated cyclic
   bilinear form   on   the coalgebra ${{{{{\mathcal A}}}}}^*$ will be  denoted $v_\theta$.

\subsection{Examples}\label{exam}   1. The group algebra ${{{{\mathcal A}}}}= \kk[G]$ of a finite group $G$   has a trace-like  $\theta \in {{{{\mathcal A}}}}^*$ which  carries the neutral element of  $ G$ to 1 and all other elements of $G$ to $0$.
  The   cyclic   form  $v_\theta$ on ${{{{\mathcal A}}}}^*$ is   the  bilinear form of Section~\ref{exaafral}    determined by     the function $\theta\vert_G:  G\to \kk$.
  
  2. In analogy with quartenions we define a unital algebra ${\mathcal A}=\kk 1_{\mathcal A}\oplus \kk i\oplus \kk j\oplus  \kk k$ with unique (associative) multiplication such that  $i^2=j^2=k^2=ijk=-1_{\mathcal A}$. The linear map  $\theta : {{{{\mathcal A}}}}\to \kk$ carrying $1_{\mathcal A}$ to $1_\kk$ and  $i,j,k$ to $0$ is trace-like. This turns ${\mathcal A}$ into a (non-commutative) symmetric Frobenius algebra and yields a cyclic bilinear form $v_\theta$ on the dual coalgebra ${{{{\mathcal A}}}}^*$.
  
 3.  For any $n\geq 1$,   the truncated polynomial algebra ${{{{\mathcal A}}}}= \kk[x]/x^{n+1}$   has a trace-like   $\theta \in M={{{{\mathcal A}}}}^*$ defined by  $\theta(  x^n)=1$ and $\theta(  x^i)=0$ for $i=0,1,..., n-1$.     Let $( u_{i })_{i =0}^n$ be the basis of   $M$ dual to the basis $( x^{i })_{i =0}^n$ of ${{{{\mathcal A}}}}$. The comultiplication ${{\mu}}$ in $M$, the cyclic structure $M\to  {{{{{\mathcal A}}}}}, \alpha \mapsto \overline \alpha$,  and the     form $ v_\theta$ on $M$ are   computed by
   $$ {{\mu}} ( u_{i })=\sum_{0\leq k \leq i}  u_{k}\otimes  u_{i-k}, \quad \overline{u_i}=x^{n-i}, \quad  v_\theta   ( u_{i } \otimes  u_{j })= \delta_{i+j, n}  $$ for all $i,j  \in \{0, 1,...,n\}$.  Given a graded algebra $A$, we write $a_{i }$ for the generator $a_{ u_{i }}  $ of   $A_M$. The       bracket $\bracket{-}{-}_v$ in $A_M$  derived from  a double bracket $\double{-}{-}$  in $A$ is computed from   \eqref{mainf}: for any $a,b\in A$ and $i,j  \in \{0, 1,...,n\}$,
$$\bracket{a_{i }}{b_{j}}_v=  \sum_{0\leq k \leq i+j-n} \double{a}{b}'_{k } \double{a}{b}''_{i+j-n-k}. $$
 It is understood that if $i+j-n<0$, then the right-hand side is equal to zero.

4.  In Example \ref{newexa}.1,  the    trace   of matrices     $\theta=\sum_i \tau_{i,i}  $ is    a  trace-like element of~$M$.      The associated cyclic  form $v_\theta$ on $M$ is computed by $v_\theta  (\tau_{i,j} \otimes \tau_{k,l})= \delta_{il} \delta_{jk}  $ for all $i,j,k,l \in \{1,...,N\}$ where $\delta$ is the Kronecker delta.
Given a graded algebra $A$ with a double bracket $\double{-}{-}$, the       bracket $\bracket{-}{-}_v$ in $A_M$    is computed from   \eqref{mainf}:
 for any $a,b\in A$ and $i,j,k,l$,
$$\bracket{a_{i,j}}{b_{k,l}}_v=   \double{a}{b}'_{k,j} \double{a}{b}''_{i,l}. $$
 This bracket   was first introduced by  Van den Bergh \cite{VdB}.

5.  In generalization of the previous example, consider a symmetric Frobenius algebra $({{{{\mathcal A}}}}, (-,-))$ with cyclic structure ${{{{\mathcal A}}}}^*\to {{{{\mathcal A}}}}, \alpha \mapsto \overline{\alpha}
$  and  with   cyclic bilinear form  $v$ on  $  {{{{{\mathcal A}}}}}^* $  as  in Theorem \ref{CQPS+}.    For any  integer $N\geq 1$, the matrix algebra $\Mat_N({{{{\mathcal A}}}})$ acquires a symmetric Frobenius   pairing
$$(a,b)=\sum_{i,j=1}^N \, (a_{ij}, b_{ji}) \in \kk$$
where $a=(a_{ij})_{i,j } , b=(b_{ij})_{i,j } \in \Mat_N({{{{\mathcal A}}}})$. This determines   a cyclic structure $\alpha \mapsto \overline{\alpha}
$ and a cylic bilinear form, $v_N$, on the coalgebra $M=(\Mat_N({{{{\mathcal A}}}}))^*$.  We   identify elements of $M$ with matrices $
\alpha=(\alpha_{ij})_{i,j}  $ where $i, j\in \{1,..., N\}$, $\alpha_{ij} \in {{{{\mathcal A}}}}^*$ for all $i,j$, so that
$\alpha (a)=\sum_{i,j} \alpha_{ij} (a_{ij})$ for any   $a=(a_{ij})_{i,j}   \in \Mat_N({{{{\mathcal A}}}})$. It is easy to see that $\overline \alpha=
(\overline{
\alpha_{ji}})_{i,j} $ and $v_N(\alpha \otimes \beta)= \sum_{i,j} v(\alpha_{ij} \otimes \beta_{ji})$ for  any $\alpha, \beta\in M$.

 \subsection{Remark}  Each 2-dimensional Topological Quantum Field Theory (TQFT) gives rise to a  unital commutative Frobenius algebra, see   \cite{Ko}.   Combining with Theorem~\ref{CQPS+}, we conclude that every 2-dimensional TQFT gives rise to a cyclic bilinear form on a coalgebra.

 \section{A group action on $A_M$}\label{mainactionsii}

 We show that coalgebra automorphisms of $M$  act  on    $A_M$ by algebra automorphisms and  study the behavior of the bracket \eqref{mainf} under this action. For counital $M$, we exhibit a group  of automorphisms of $A_M$ preserving the bracket.

\subsection{The group $  \Aut(M)$}\label{thegrnsAN+++}      A {\it coalgebra automorphism}  of  a coalgebra $M=(M,{{\mu}} )$  is an invertible linear map  $\omega  :M\to M$ such that ${{\mu}}     \omega  =(  \omega   \otimes   \omega  ){{\mu}}$.  Let $ {\Omega}= \Aut(M)$  be the group of all coalgebra automorphisms of $M$ with multiplication $\omega  \omega'  = \omega  \circ \omega'   $ for   $\omega  , \omega'   \in {\Omega}$.

 \begin{lemma}\label{actonbilformn}
 The group ${\Omega}$ acts  on $(M\otimes M)^* $ on the right by   $v\mapsto v^\omega    =v (\omega   \otimes \omega  ) $ for   $\omega  \in {\Omega}$ and  $v\in (M\otimes M)^*$. This action preserves the set of cyclic bilinear forms.
\end{lemma}

\begin{proof} The first claim is obvious. It is easy to check  that for any $v\in (M\otimes M)^*$,
$$\widehat{v^\omega  }= (\omega  \otimes \omega  )^{-1} \widehat v (\omega   \otimes \omega  )\in \End (M\otimes M).$$
If $v$ is cyclic, then    $p_{21}$ commutes with both $\widehat v$ and $\omega  \otimes \omega  $  and hence  commutes with~$\widehat{v^\omega  }$. Hence,   $ v^\omega $ is cyclic.
\end{proof}

      For any graded algebra $A$, there is a natural   left  action of   ${\Omega}=  \Aut(M)$      on   ${\widetilde A}_{M}$. Any    $\omega  \in {\Omega}$ acts on the generators by
$ \omega    a_{\alpha}=
a_ {\omega  (\alpha)}$. This extends uniquely to a graded algebra automorphism of ${\widetilde A}_{M}$. The compatibility   with the bilinearity relations is obvious. The compatibility   with the multiplicativity relations:
$$
\omega   (ab)_{\alpha}   = (ab)_{\omega   (\alpha)}   =  a_{(\omega  (\alpha))^1 } b_{(\omega  (\alpha))^2}    =
 a_{\omega  (\alpha^1)} b_{\omega  (\alpha^2)}   =
  \omega   a_{ \alpha^1 } \cdot  \omega   b_{ \alpha^2 }
  =  \omega   (a_{ \alpha^1 }  b_{ \alpha^2 })
$$ where the third equality holds because ${{\mu}}     \omega  =(  \omega   \otimes   \omega  ){{\mu}}$. The   action of ${\Omega}$ on ${\widetilde A}_{M}$    induces    an action  of ${\Omega}$   on $A_M=\Com({\widetilde A}_{M})$ by graded algebra automorphisms.


  \begin{lemma}\label{fdbtopb+action}
For any cyclic bilinear form $v$ on $M$, any $\omega  \in {\Omega}  $
and   $x,y\in   A_M$, we have $\bracket{\omega  x }{\omega  y}_v=  \omega
\bracket{x} {y}_{v^\omega  }$. Therefore,   the bracket $\bracket{- }{-}_v$ in $A_M$ is preserved under the action of the isotropy group $\{\omega  \in \Omega\, \vert \, v^\omega  =v\} $ of $v$.
\end{lemma}

\begin{proof}   It is easy to see that if
the identity    $\bracket{\omega  x }{\omega  y}_v=  \omega
\bracket{x} {y}_{v^\omega  }$
holds for   generators   of $ A_M$,
then it holds for any $x,y \in   A_M$.   Given  $a,b\in A$ and $\alpha, \beta \in M$,
\begin{eqnarray*}
\bracket{ \omega   a_\alpha}{\omega   b_{\beta} }_v
&=& \bracket{a_{\omega  (\alpha)} }{b_{\omega  (\beta)}}_v\\
&= &  v(\omega  (\alpha) \otimes  \omega  (\beta)^2)  \double{a}{b}'_{\omega  (\beta)^1}  \double{a}{b}''_{\omega  (\beta)^3} \\
&=& v(\omega  (\alpha) \otimes  \omega  (\beta ^2))  \double{a}{b}'_{\omega  (\beta^1)}  \double{a}{b}''_{\omega  (\beta^3)}\\
&=& v^\omega  ( \alpha  \otimes   \beta ^2)  \omega  (\double{a}{b}'_{ \beta^1 })  \omega  (\double{a}{b}''_{ \beta^3 })\\
&=& \omega  ( v^\omega  ( \alpha  \otimes   \beta ^2)   \double{a}{b}'_{ \beta^1 } \double{a}{b}''_{ \beta^3 })=    \omega   \bracket{   a_\alpha}{  b_{\beta} }_{v^\omega  }.
  \end{eqnarray*}
  \end{proof}

 \subsection{The counital case}  For a counital coalgebra  $M$, one   defines  \emph{inner    automorphisms}  as follows.
Note  that  the counit $\varepsilon=\varepsilon_M\in M^*$   is a two-sided unit of  the  algebra   $M^*$ dual to $M$. Since $M^*$ has a unit, we can consider  the group   $U=U(M^*)$  of invertible elements of $M^*$.  Each $u\in U$ acts on $M$ by the  inner    automorphism       $\alpha\mapsto  {}^u\alpha= u^{-1} (\alpha^1) \alpha^2 u(\alpha^3)$.
  This   defines   a  left  action of   $U$  on $M$ by coalgebra automorphisms, that is  a   homomorphism $U \to \Aut(M)$.
 Composing   with the actions  of $\Aut(M)$ on $\widetilde A_M$, $A_M$   we obtain    actions   of $U $ on $\widetilde A_M$, $A_M$.

 Note for the record that   the action of any $u\in U$ on $M$ is dual to the    conjugation by $u$ in $M^*$.  Indeed, for any  $a\in M^*$ and $\alpha \in M$,
 \begin{equation}\label{duac}  a( {}^u\alpha )=a(u^{-1} (\alpha^1) \alpha^2 u(\alpha^3))= u^{-1} (\alpha^1) a(\alpha^2)  u(\alpha^3)= (u^{-1} au)(\alpha).\end{equation}

 \begin{theor}\label{justxxxz} Let  $v\in (M\otimes M)^*$ be a cyclic bilinear form on a counital coalgebra~$M$. Then

 (i) $v$ is symmetric;

 (ii)  the action of  $U=U(M^*)$   on $A_M$ preserves the bracket $\bracket{- }{-}_v$;

 (iii) the fixed point algebra \begin{equation*} A_M^{U}=\{ x\in A_M\, \vert\, ux=x \,\, {\rm {for \,\,  all}} \,\, u\in U\} \end{equation*}
 is closed under the bracket $\bracket{- }{-}_v$.
\end{theor}

\begin{proof}  Let  $\alpha, \beta \in M$.  Applying $\varepsilon \otimes \varepsilon: M\otimes M\to \kk$ to both sides of   \eqref{forcyc}, we obtain that  $v (\alpha \otimes \beta)=v (\beta \otimes \alpha)$. This equality may be rewritten as $\overline \alpha (\beta)= \overline \beta (\alpha)$
where the overbar denotes the cyclic structure $M\to M^*$ determined by $v$.

 By Lemma~\ref{fdbtopb+action}, in order to prove (ii), it is enough to show that $v^u=v$ for all $u\in U$.  We  first  verify that $\overline {{}^u \alpha}= u \overline \alpha u^{-1}$ for any $\alpha \in M$.
 We have $$\overline {{}^u \alpha}= \overline {u^{-1} (\alpha^1) \alpha^2 u(\alpha^3)}=u^{-1} (\alpha^1)\, \overline {\alpha^2} \,u(\alpha^3). $$ Evaluating on any $\beta \in M$, we obtain
 $$\overline {{}^u \alpha} (\beta)= u^{-1} (\alpha^1)\, \overline {\alpha^2} (\beta) \,u(\alpha^3)$$
 $$=
 u^{-1} (\alpha^1)\, \overline {\beta} (\alpha^2) \,u(\alpha^3)=
 u(\beta^1)\, \overline {\alpha} (\beta^2) \,u^{-1} (\beta^3)=(u \overline {\alpha}u^{-1}) (\beta)  $$
 where the penultimate equality is obtained by evaluating  $u\otimes u^{-1}$ on both sides of~\eqref{cycl1}. We conclude that $\overline {{}^u \alpha}=u \overline {\alpha}u^{-1}$.
 Next, identifying $u\in U$ with its image in $\Aut(M)$, we obtain for any $\alpha, \beta \in M$,
 $$v^u(\alpha \otimes \beta)=v ({}^u \alpha \otimes {}^u \beta)= \overline {{}^u \alpha} ( {}^u \beta )=(u \overline \alpha u^{-1})({}^u\beta )$$
 $$\stackrel {\eqref{duac}}=(u^{-1} u \overline \alpha u^{-1} u) (\beta )=\overline \alpha(\beta )= v(\alpha \otimes \beta).$$
 We conclude that $v^u=v$. This proves (ii). Clearly,  (ii) implies (iii).
\end{proof}


 \section{A Lie algebra action on $A_M$}\label{mainactionsiiLiee}

  We show that coderivations of   $M$  act  on    $A_M$ by derivations and study the behavior of the bracket \eqref{mainf} under this action. For    counital $M$, we exhibit a Lie algebra of derivations of $A_M$ preserving the bracket.

\subsection{Derivations}\label{algebras++}

A (degree zero) \emph{derivation}  in a graded  algebra $A$     is a linear map $\delta :A \to A$ such that $\delta(A^p)\subset
A^{p}$   for all $p\in \ZZ$ and \begin{equation}\label{deri} \delta(ab)=\delta(a)b+ a\delta(b)\end{equation}
for any    $a,b \in A$.
The derivations  of $A$   form a    module, $\Der(A)$.
We provide $\Der(A)$ with the   Lie bracket  by
$[\delta ,d]=  \delta d-   d \delta $
for any  $\delta , d\in \Der(A)$. Any derivation of   $A$  carries $[A,A]$ into itself and   induces     a derivation  of the algebra $ \Com(A) $.

 An
\emph{action} of a   (non-graded) Lie algebra  $\g$  on $A$ is a   Lie algebra homomorphism      $\g\to
\Der(A)$. Such an action     induces  an     action  of $\g$  on $  \Com(A)$ in the obvious way.

 A  bracket $\bracket{-}{-}$ in $A$ is   \emph{invariant} under a derivation $\delta:A\to A$ if $\delta (\bracket{a}{b})=\bracket{\delta (a)}{b}+\bracket{a}{\delta(b)}$ for any   $a,b\in A$.
 A  bracket $\bracket{-}{-}$ in $A$ is   \emph{invariant} under an action   of a Lie algebra  $\g$  on $A$ if it is invariant under the action of all elements of~$\g$. Note that then $\bracket{A^\g }{A^\g}\subset A^\g$ where
  \begin{equation}\label{fifi}  A^\g=  \{ a\in A \, \vert\, w (a)=0 \,\, {\rm {for \,\,  all}} \,\, w\in \g \}.\end{equation}

\subsection{Coderivations}\label{algebras++ss} A  \emph{coderivation} in a coalgebra  $M$  is a  linear map $\delta:M\to M$ such that \begin{equation}\label{coco} {{\mu}} \delta=(\delta \otimes \id_M+\id_M \otimes \delta){{\mu}}.\end{equation} This  formula may be rewritten as the identity $$\delta (\alpha)^1     \otimes {\delta (\alpha)^2}  =  \delta (\alpha^1   ) \otimes {\alpha^2}  +  {\alpha^1} \otimes {\delta (\alpha^2   )} $$ for any $\alpha\in M$.  The coderivations  of $M$   form a    Lie algebra, $\coder(M)$,  with the   Lie bracket
$[\delta ,d]=  \delta d-   d \delta   $
for any  $\delta , d\in \coder(M)$.

The  Lie algebra $\coder (M) $ contains a   Lie  subalgebra of inner coderivations.   Namely, consider the algebra $M^* $ dual to $M$, and let $\underline M^*$ be the  module $M^*$   equipped with the  Lie bracket $[\varphi,\psi ]=\varphi\psi  -\psi  \varphi$ for any $\varphi,\psi  \in M^*$. Every   $\varphi\in   M^*$ determines an {\it inner  coderivation} $\delta_\varphi$ of $M$   by $\delta_\varphi(\alpha)=    \varphi (\alpha^2) \alpha^1 -\varphi (\alpha^1) \alpha^2$ for  $\alpha \in M$. The map $\underline M^* \to  \coder (M),  \varphi\mapsto \delta_\varphi$ is a   Lie algebra homomorphism.

\begin{lemma}\label{acti} For a graded algebra $A$ and   a coalgebra $M$, there
  is a unique      action of  the   Lie algebra $\coder (M) $    on
${\widetilde A}_{M}$ such that $ \delta (a_{\alpha})=   a_{\delta(\alpha)  }$ for any    $\delta \in \coder (M)$, $a\in A$, $\alpha \in {M}$.
\end{lemma}

\begin{proof} The uniqueness is obvious, and we need only to construct the action in question.   We derive from each $\delta \in \coder (M)$ a derivation $\widetilde \delta$ in $\widetilde A_M $  such that $\widetilde \delta (a_{\alpha})=   a_{\delta(\alpha)  }$ for all  generators $a_\alpha$ of $\widetilde A_M$.
We  check the compatibility with the  defining  relations of $\widetilde A_M$.   The compatibility with the bilinearity  relations is obvious.  The compatibility with the  multiplicativity relations:
\begin{eqnarray*} \widetilde  \delta ((ab)_{\alpha})  & {=}& (ab)_{  \delta(\alpha)  }    = a_{\delta(\alpha)^1}   b_{\delta(\alpha)^2}  \\
&=&  a_{ \delta(\alpha^1   )}    b_{\alpha^2}  +   a_{\alpha^1}  b_{\delta(\alpha^2   )}\\
&=&  \widetilde \delta(a_{\alpha^1}) b_{\alpha^2} +   a_{\alpha^1} \widetilde \delta(b_{\alpha^2})= \widetilde \delta(a_{\alpha^1} b_{\alpha^2}).
 \end{eqnarray*}

The  map $\coder (M) \to \Der (\widetilde A_M), \delta \mapsto \widetilde \delta$  is linear and preserves the Lie bracket: for    $\delta,d \in \coder (M)$,   $a\in A$, $\alpha \in {M}$,
$$[\widetilde \delta, \widetilde d] (a_{\alpha})   {=}  \widetilde \delta \widetilde d (a_{\alpha})-   \widetilde d  \widetilde \delta (a_{\alpha})
=  a_{  \delta d(\alpha)  }-      a_{ d \delta(\alpha)  }
=  a_{ [ \delta,d](\alpha)  } =\widetilde {[ \delta,d]} (a_{  \alpha   })  .$$
Since $\{a_\alpha\}$ generate  the algebra $\widetilde A_M$, we have $[\widetilde \delta, \widetilde d]=\widetilde {[ \delta,d]} $.
\end{proof}

   Composing  the map $\underline M^* \to  \coder (M),  \varphi\mapsto \delta_\varphi$  with the action of $\coder (M)$ on~$\widetilde A_M$,      we obtain an action of the Lie algebra  ${\underline M}^*$ on $\widetilde A_M$. This    induces an     action of   ${\underline M}^*$   on $A_{M}=\Com(\widetilde A_{M})$. We   state an analogue of Theorem~\ref{justxxxz}  in this context.

\begin{theor}\label{justxxxzLie} For any double bracket $\double{-}{-}$ in a graded algebra $A$ and any cyclic bilinear form $v$ on a counital coalgebra~$M$,
   the induced  bracket $ \bracket{- }{-}_v$ in $A_M$ is invariant under the action of   ${\underline M}^*$. As a consequence, the  graded algebra \begin{equation}\label{Afix} A^M_M=(A_M)^{{\underline M}^*}= \{ x\in A_M\, \vert\, \delta_\varphi( x)=0 \,\, {\rm {for \,\,  all}} \,\, \varphi\in {\underline M}^*\}\end{equation}
is   closed under the bracket $\bracket{- }{-}_v$.
\end{theor}

\begin{proof}     We verify first that   the   bracket $ \bracket{- }{-}_v$ is invariant under the action   on $A_M$ of any coderivation  $\delta:M\to M$ such that $v(\delta  \otimes \id_M)=- v(\id_M \otimes \delta )$.   We must prove that $ \delta (\bracket{x}{y}_v)= \bracket{\delta  (x)}{y}_v+\bracket{x}{\delta  (y)}_v$ for all     $x,y \in A_M$. A simple computation using the Leibniz rules shows that it is enough to prove this identity for the generators $x=a_\alpha$, $y=b_\beta$ of $A_M$ where $a,b\in A$, $\alpha, \beta\in M$. We have
$$ \bracket{\delta  (a_\alpha)}{b_\beta}_v+\bracket{a_\alpha}{\delta  (b_\beta)}_v= \bracket{  a_{\delta (\alpha)}}{b_\beta}_v+\bracket{a_\alpha}{ b_{\delta  (\beta)}}_v$$
$$= \double{a}{b}'_{\delta (\alpha)_\beta } \double{a}{b}''_{\beta^{\delta(\alpha)}}+\double{a}{b}'_{ \alpha_{\delta (\beta)}} \double{a}{b}''_{\delta (\beta)^\alpha }.$$
On the other hand,
$$\delta ( \bracket{a_\alpha}{b_\beta})= \delta (\double{a}{b}'_{\alpha_\beta} \double{a}{b}''_{\beta^\alpha})
=\double{a}{b}'_{\delta (\alpha_\beta)} \double{a}{b}''_{\beta^\alpha}+\double{a}{b}'_{ \alpha_\beta} \double{a}{b}''_{\delta (\beta^\alpha)}.$$
We only need to prove that for all $\alpha, \beta\in M$, $$  {\delta (\alpha)_\beta } \otimes {\beta^{\delta(\alpha)}}+ { \alpha_{\delta (\beta)}} \otimes {\delta (\beta)^\alpha } ={\delta (\alpha_\beta)} \otimes {\beta^\alpha}+ { \alpha_\beta}\otimes {\delta (\beta^\alpha)} $$
or, equivalently,
\begin{equation}\label{idh} \widehat v (\delta \otimes \id_M+ \id_M \otimes \delta)=(\delta \otimes \id_M+ \id_M \otimes \delta) \widehat v.\end{equation}
Since $\delta$ is a coderivation, we easily compute that for any $\beta \in M$, $${{\mu}}^2 (\delta(\beta))=  \delta(\beta^1) \otimes \beta^2 \otimes \beta^3+
\beta^1 \otimes \delta(\beta^2) \otimes \beta^3+ \beta^1 \otimes  \beta^2  \otimes   \delta(\beta^3).$$
Hence, evaluating the left-hand side  of \eqref{idh} on $\alpha \otimes \beta  \in M\otimes M $, we obtain
$$ v(\delta (\alpha) \otimes \beta^2) \, \beta^1  \otimes \beta^3
+v(\alpha \otimes  \beta^2) \, \delta(\beta^1)   \otimes  \beta^3 $$
$$+v(\alpha \otimes  \delta(\beta^2)  ) \, \beta^1 \otimes \beta^3 +v(\alpha \otimes  \beta^2) \, \beta^1 \otimes  \delta(\beta^3)  .$$
  The first and  third terms cancel out by the assumption $v(\delta  \otimes \id_M)=- v(\id_M \otimes \delta )$. The remaining two terms
   give
   \begin{equation*}\label{idh+}v(\alpha \otimes \beta^2)( \delta (\beta^1) \otimes \beta^3 +\beta^1 \otimes \delta(\beta^3))=(\delta \otimes \id_M+ \id_M \otimes \delta) \widehat v (\alpha \otimes \beta).\end{equation*}   This proves \eqref{idh}.

  To accomplish the proof, it remains to show that $v(\delta_\varphi  \otimes \id_M)=- v(\id_M \otimes \delta_\varphi )$ for all $\varphi\in M^*$. We evaluate both sides on any $\alpha \otimes \beta  \in M\otimes M $. Recall the cyclic structure
$M \to M^*, \gamma\mapsto \overline \gamma$ determined by $v$.    Applying    $\id_M \otimes \varepsilon_M$ and $  \varepsilon_M  \otimes  \id_M $  to both sides of \eqref{cycl1} we obtain   \begin{equation}\label{fghj}    {\overline \beta} (\alpha^1) \alpha^2   =   {\overline \alpha} (\beta^2)   \beta^1  \quad {\rm{and}}\quad   {\overline \beta} (\alpha^2) \alpha^1   =   {\overline \alpha} (\beta^1)   \beta^2 .\end{equation}
Using  these formulas and the identity   $  {\overline \gamma} (\beta )={\overline \beta} (\gamma)    $ for all $\gamma\in M$, we obtain
\begin{equation*}\label{strange}  v(\delta_\varphi (\alpha)  \otimes
  \beta)= \overline{\delta_\varphi(\alpha)} (\beta)= \overline{\beta} (\delta_\varphi(\alpha))
  = \varphi(\alpha^2)  \overline{\beta} (\alpha^1)- \varphi(\alpha^1)\overline{\beta} ( \alpha^2)\end{equation*}
\begin{equation*} =  \varphi(\beta^1)  \overline{\alpha} (\beta^2)- \varphi(\beta^2)\overline{\alpha} ( \beta^1)  =
  - \overline{\alpha} (\delta_\varphi(\beta))=- v(\alpha \otimes \delta_\varphi(\beta)) . \qedhere \end{equation*}
\end{proof}

\subsection{Remark} Subtracting the first of the equalities \eqref{fghj} from the second one, we obtain the following useful identity: for all $\alpha, \beta \in M$,
\begin{equation}\label{fghje} \delta_{\overline \beta}(\alpha)= -\delta_{\overline \alpha}(\beta). \end{equation}

\subsection{Compatibility}\label{compa} To relate  the   actions of $\Aut(M)$ and $\coder (M) $ constructed above, we use the language of Lie pairs introduced in \cite{MT2}. A {\it Lie pair} is a pair $(G,\g)$ where $G$ is a group and $\g$ is a   Lie algebra equipped with a (left) action of $G$ by Lie algebra automorphisms $w \to  {}^g  w$ where $w $ runs  over $\g$ and $g$  runs  over $G$. A {\it morphism} of Lie pairs $(G',\g')\to (G,\g)$
is a pair (a group homomorphism $F:G'\to G$, a Lie algebra homomorphism $f:\g'\to \g$) such that $f({}^g  w)= {}^{F(g)}  (f(w))$ for all $g\in G'$, $w\in \g'$. An {\it action} of a Lie pair $(G,\g)$   on a graded algebra $A$ is a pair (a (left) action of $G$ on $A$ by  graded algebra automorphisms,  an action of    $\g$  on $A$) such that  $({}^g  w)x  = g  w  g^{-1} x$ for any $g\in G$,   $w \in \g $,   $x\in A$. Note that then the   algebra $A^{\g} $ defined by \eqref{fifi} is $G$-invariant. Indeed, if $ g\in G$ and $x\in A^{\g}$, then $gx\in A^{\g}$ because for all $ w \in \g$,
$$w  gx =g(g^{-1} w  g)x =g ({}^{g^{-1}}  w)  x =g0=0  .$$
 An action  of   $(G,\g)$  on   $A$ composed with a morphism  of Lie pairs $(G',\g')\to (G,\g)$ yields an action of $(G',\g')$ on $A$. An action  of a Lie pair    on   $A$  induces an action  of this Lie pair on $\Com(A)$ in the obvious way.

Any coalgebra $M$ gives rise to a Lie pair $(G=\Aut(M), \g=\coder   (M) )$ where $G$ acts on $\g$ by   ${}^g  w= g  w   g^{-1} :M\to M$ for   $g\in G$ and   $w \in \g $.  The above-defined actions of $G$ and $\g $ on $\widetilde A_M$
determine  an action  of the Lie pair $(G,\g)$ because $({}^g  w)x  = g  w   g^{-1} x$ for any $g\in G$,   $w \in \g $,   $x\in A$.
Since   ${}^g  w$ and $ g  w   g^{-1} $ act as derivations in $\widetilde A_M$,  it is enough to check this    for  the generators $x=a_\alpha\in \widetilde A_M$ where $a\in A$, $\alpha\in M$. We have
$${}^g w( a_\alpha)=a_{({}^g w(  \alpha))}=a_{g  w   g^{-1} (\alpha)}=  g a_{   w g^{-1}    (\alpha)}=  g w (a_{    g^{-1}  (\alpha)}) = g w g^{-1} (a_\alpha). $$

A counital coalgebra $M$ gives rise to a Lie pair $( U ,  \underline {M}^* )$ where $U=U(M^*)$ acts on $\underline {M}^*$ by      ${}^u  \varphi  =   u\varphi u^{-1}   $ for   $u\in U$, $\varphi\in M^*$.  It is easy to check that the above-defined homomorphisms $U\to \Aut(M)$ and
$\underline {M}^*\to \coder   (M)$ determine a morphism of Lie pairs. Therefore the above-defined actions of $U$ and $\underline {M}^* $ on $\widetilde A_M$
determine an action  of the Lie pair $( U,  \underline {M}^* )$ on $\widetilde A_M$. Considering   the induced action on $A_M=\Com(\widetilde A_M)$, we can conclude  that, under the assumptions of Theorem~\ref{justxxxzLie}, the  graded algebra $ A^M_M \subset A_M$ is $U$-invariant.

\section{The bracket $\langle - , - \rangle$ and the traces}\label{trace}

    A double bracket in   $A$ induces a bracket  $\langle - , - \rangle$ in $\check A=A/[A,A]$, and we relate it    to the bracket  $\bracket{-}{-}_v$ in $A_M$.

\subsection{The bracket $\langle - , - \rangle$}\label{stdb}

 An $n$-graded double  bracket $\double{-}{-}$ in a graded algebra $A$ induces  a bilinear form $ A\times A\to A, (a,b)\mapsto  {\double{a }{b}}'  {\double{a }{b}}'' $.  The latter descends to an $n$-graded antisymmetric    bracket $ \langle -,-  \rangle: \check A \times \check A \to \check A$   in the graded module  $\check A=A/[A,A]$, see \cite{VdB}.   If   $\double{-}{-}$ is Gerstenhaber, then  $ \langle -,-  \rangle$ is   an $n$-graded Lie bracket, i.e.,  it is $n$-graded antisymmetric and satisfies the $n$-graded  Jacobi identity.

\subsection{Traces}\label{subs-trace}  Every  element $\theta$ of a coalgebra $ M$  determines a degree-preserving linear map $ \tr_\theta:   A\to A_M$   by $\tr_\theta (a)=a_\theta$ for   $a\in A$. The subalgebra of $A_M$ generated by the set $\tr_\theta(  A)=\{a_\theta\}_{a\in A}$ is   denoted   $A(\theta)$.

We call $\theta\in M$ \emph{symmetric} if
${{\mu}}(\theta)\in M\otimes M$ is invariant under the permutation of the tensor factors, i.e., if $p_{21} {{\mu}} (\theta)= {{\mu}} (\theta)$.

 \begin{lemma}\label{trat--} For a  symmetric $\theta\in M$,  we have $[A,A]\subset \Ker \tr_\theta$ and $ A(\theta) \subset A_M^M $ where    $A^M_M$ is the subalgebra of $A_M$ defined by \eqref{Afix}.
 \end{lemma}

\begin{proof}  For any homogeneous $a, b\in A$,
\begin{eqnarray*}
\tr_\theta (ab- (-1)^{\vert a \vert \, \vert b\vert} ba) &=&  (ab)_\theta- (-1)^{\vert a \vert \, \vert b\vert} (ba)_\theta \\
& {=}&   a_{\theta^1} b_{\theta^2} - (-1)^{\vert a \vert \, \vert b\vert} b_{\theta^1}a_{\theta^2}
=a_{\theta^1} b_{\theta^2}- a_{\theta^2} b_{\theta^1}=0
\end{eqnarray*}
where the last   equality holds because $p_{21} {{\mu}} (\theta)= {{\mu}} (\theta)$.
Thus, $\tr_\theta([A,A])=0$.

The equality $p_{21} {{\mu}} (\theta)= {{\mu}} (\theta)$   implies that each inner coderivation $\delta$ of $M$ annihilates $\theta$. Therefore $\delta (a_\theta)= a_{\delta (\theta)}=0$ for all $a\in A$. Hence $ \tr_\theta(  A) \subset A_M^M $.
\end{proof}

  We   say that   $\theta \in M$  is \emph{adjoint} to a bilinear form $v$ on   $M$     if the map $M\to \kk$, $\alpha \mapsto v(\theta \otimes \alpha)$ is a counit of $M$.

 \begin{lemma}\label{trat}  Any   $\theta \in  M$   adjoint to    a cyclic bilinear form $v$ on   $M$  is symmetric.
 \end{lemma}

\begin{proof} Since the map $\alpha \mapsto v(\theta \otimes \alpha)=\varepsilon (\alpha)$ is the counit of $M$, \begin{equation}\label{simmth---}\widehat v (\theta \otimes \theta)= v(\theta \otimes \theta^2) \theta^1\otimes \theta^3=\varepsilon (\theta^2)
\theta^1\otimes \theta^3=\theta^1\otimes  \theta^2={{\mu}}(\theta).\end{equation}
Therefore $\theta$  is symmetric:
\begin{equation*}\label{simmth} p_{21} {{\mu}} (\theta) = p_{21} \widehat v (\theta \otimes \theta)= \widehat v p_{21} (\theta \otimes \theta)=  \widehat v (\theta \otimes \theta)= {{\mu}}(\theta). \qedhere\end{equation*}
\end{proof}

 By Lemma~\ref{trat--}, the   map $ \tr_\theta:   A\to A_M$  associated with a   symmetric $\theta\in M$ induces a  linear map   $ \check A \to  A_M^M$. It is called the {\it trace} and  also   denoted  by $\tr_\theta $.

 \begin{lemma}\label{trat++} For a  double bracket $\double{-}{-}$ in a graded algebra $A$, a  cyclic bilinear form $v $ on a   coalgebra $M$, and any $\theta\in M$  adjoint to $v$,     the trace   $   \tr_\theta: \check A\to A_M$  carries  the bracket   $\langle- ,- \rangle$  in $  \check A$   to  the
bracket $\bracket{-}{-}_v$ in $A_M$. Consequently, the algebra $A(\theta)$ is closed under the
bracket $\bracket{-}{-}_v$.
 \end{lemma}

\begin{proof}   Pick   any   $a,b\in A$ and let $\check a, \check b$ be their projections to $\check A$.  Then
\begin{eqnarray*}
\bracket{\tr_\theta(\check a)}{\tr_\theta(\check b)}_v &=&   \{ \,   a_\theta,   b_\theta \, \}_v \\
& \stackrel{\eqref{simmth---}}{=}&     \double{a}{b}'_{ \theta^1} \double{a}{b}''_{ \theta^2}\\
&=&     \left(\double{a}{b}' \double{a}{b}''\right)_{\theta}\\
&=&  \tr_\theta \left(\double{a}{b}' \double{a}{b}''\right)
\  {=} \   \tr_\theta\big(\langle \check a, \check b \rangle\big).
 \quad \quad \quad \quad \quad \qedhere   \end{eqnarray*}
\end{proof}

The following lemma gives  more information about   $ A(\theta)$   for some $\theta$.

 \begin{lemma}\label{tratxxx}    Any  non-degenerate cyclic  bilinear  form $v$  on a counital coalgebra  $M$ has a unique adjoint $\theta\in M$  and  $ A(\theta) \subset   A^{M}_M\cap A_M^{U}$ where $U=U(M^*)$.
 \end{lemma}

\begin{proof} The first claim is obvious. To prove the second claim, we
  must show that ${}^u (a_\theta)
 =a_\theta$     for all $u\in U  $ and $a\in A$. Pick any $\alpha \in M$ and observe that
$$v({}^u\theta \otimes  \alpha)= v({}^u\theta \otimes {}^u {}^{u^{-1}} \alpha)=v^u(\theta\otimes {}^{u^{-1}} \alpha)=v (\theta\otimes {}^{u^{-1}} \alpha)$$
$$   = \varepsilon_M( {}^{u^{-1}} \alpha ) \stackrel{\eqref{duac}}=u \varepsilon_M u^{-1}(\alpha )=
\varepsilon_M (\alpha     )=v( \theta \otimes  \alpha) $$
where   we use that   $v^u=v$. The injectivity of $\ad_v$ implies  that  ${}^u\theta=\theta$. Hence,   $ {}^u (a_\theta)
=a_{({}^u\theta)}=a_\theta $.
\end{proof}


The next claim yields  a rich source of adjoint elements of cyclic bilinear forms.

 \begin{lemma}\label{tratxxxbbb}  Let ${{{{\mathcal A}}}}$  be a unital   algebra whose underlying module is free of finite rank. Let   $ M={{{{\mathcal A}}}}^*$ be the dual coalgebra and let $\theta\in M$ be a trace-like element. Then $\theta$ is adjoint to the cyclic bilinear form $v_\theta$ on $M$ defined in Section~\ref{frofro}. \end{lemma}

\begin{proof}  Consider the symmetric Frobenius algebra $({{{{\mathcal A}}}}, (-,-)_\theta)$ as  in Section~\ref{frofro}. The  cyclic structure $M\to {{{{\mathcal A}}}}, \alpha \mapsto \overline \alpha $ defined in Theorem~\ref{CQPS+}  carries $\theta$ to $1_{{{{{\mathcal A}}}}}$. Indeed,   for all $a\in  {{{{\mathcal A}}}}$,
$$\theta (a  \overline \theta) =(a, \overline \theta)_\theta =\theta (a) = \theta (a 1_{{{{{\mathcal A}}}}})$$
and so $ \overline \theta=1_{{{{{\mathcal A}}}}}$. For any $\alpha\in M$, we have $v_\theta (\theta \otimes \alpha)=\alpha (\overline \theta) = \alpha (1_{{{{{\mathcal A}}}}})$.
Therefore $\theta$ is adjoint to   $v_\theta$. \end{proof}

 \section{The unital case}\label{remremrem}

 For a unital graded algebra $A$ and a counital coalgebra $M$, we define unital versions $\widetilde A_M^+$, $A_M^+$ of   $\widetilde A_M$, $A_M$. We also
 formulate    Hamiltonian reduction and discuss double quasi-Poisson algebras.

 \subsection{Unital representation algebras}
 Unital graded algebras and graded algebra  homomorphisms   carrying $1$ to $1$
 form a category ${\grAlg}^+$. For a unital graded algebra~$ A$ and a counital coalgebra~$M$ with counit $\varepsilon=\varepsilon_M$,  we define a unital graded algebra  $\widetilde A_M^+$    as follows. First, we adjoin a two-sided unit   to $\widetilde A_M$, that is consider the unital graded algebra $ \kk e\oplus  \widetilde A_M$ with two sided unit~$  e $ of degree zero.  A typical element of $\kk e\oplus  \widetilde A_M$ is represented by a   non-commutative polynomial in the generators $\{a_\alpha \, \vert \, a\in A, \alpha\in M\}$. By definition, $\widetilde A_M^+$  is  the  quotient of the algebra  $\kk e\oplus  \widetilde A_M$  by   the      normalization relations  $\{(1_A)_{\alpha}=\varepsilon   ( \alpha ) e\}_{\alpha\in M}$. The construction of $ \widetilde A^+_{M}$    obviously extends to a   functor $f\mapsto  \widetilde f^+_M:  {\grAlg}^+ \to {\grAlg}^+$.


 \begin{lemma}\label{just+}
For any unital graded algebra $B$,  the   convolution algebra $H_M(B) $ is unital with  unit $M\to B, \alpha \mapsto \varepsilon   ( \alpha ) 1_B$. The bijection \eqref{eq:adjunction} induces   a   bijection
\begin{equation}\label{eq:adjunctionkkk}
\Hom_{{\grAlg}^+} ( \widetilde A_{M}^+ , B)
\stackrel{\simeq}{\longrightarrow} \Hom_{{\grAlg}^+} (A,  H_M(B))
\end{equation}
which is natural in $A$ and $B$.
\end{lemma}

\begin{proof} The first claim is obvious. To prove the second claim, we apply the same constructions $r\mapsto s_r$ and $s\mapsto r_s$ as in the proof of Lemma~\ref{just}.
Observe that the map  $s_r:A\to \Hom_\kk (M,B)$  carries $1_A$ to
  the map $$M\to B,\, \alpha \mapsto  r( (1_A)_{\alpha }) = r(\varepsilon  ( \alpha ) e) = \varepsilon  ( \alpha ) r(e)= \varepsilon  ( \alpha ) 1_B$$
which is   the unit of the algebra $
  H_M(B)$. The map $r=r_s$ is compatible with the normalization relations because  \begin{equation*} r( (1_A)_{\alpha })= s(1_A)(\alpha) = \varepsilon  ( \alpha ) 1_B
=r( \varepsilon  ( \alpha ) e ). \qedhere \end{equation*}
\end{proof}

Replacing $ \widetilde A_M^+$   by   $A_M^+=\Com(\widetilde A_M^+)$ we obtain that for any unital commutative graded algebra $B$,
\begin{equation}\label{uncogral}
 \Hom_{\grCom^+} ({  A}_M^+, B) {\simeq}  \Hom_{\grAlg^+} (\widetilde A_{M}^+, B)
{\simeq} \Hom_{\grAlg^+} (A, H_M(B)) 
\end{equation}
where $\grCom^+$ is the category of unital
commutative  graded algebras. The   algebra $A_M^+$ can be alternatively  obtained by adjoining a   two-sided unit $e$  to   $  A_M=\Com(\widetilde A_M)$ and then factoring $ \kk e\oplus    A_M$ by the   relations $\{(1_A)_{\alpha}=\varepsilon   ( \alpha ) e\}_{\alpha\in M}$.
The construction of $   A^+_{M}$  obviously   extends to a functor  $f\mapsto   f^+_M: {\grAlg}^+\to  \grCom^+$.

\subsection{The   actions} Given a unital graded algebra $A$ and a counital coalgebra $M$, the action of     $\Aut(M)$ on $\widetilde A_M$ defined in  Section~\ref{thegrnsAN+++}     induces  an action  of $\Aut(M)$   on  $\widetilde A_M^+$   by graded algebra automorpisms fixing the unit~$e$.
 We check the compatibility   with the normalization relations: for $g\in \Aut(M)$, $\alpha\in M$,
$$g(1_A)_\alpha=(1_A)_{g(\alpha)} = \varepsilon ( g (\alpha) ) e = \varepsilon (  \alpha)  e  =g (\varepsilon (  \alpha)  e) .$$
Similarly, the action of the Lie algebra  $\coder (M) $ on $\widetilde A_M$ defined in  Lemma~\ref{acti}   induces  an action  of $\coder (M) $   on  $\widetilde A_M^+$    annihilating    $e$. The compatibility   with the normalization relations holds because  $\varepsilon \delta=0$ for all $\delta\in \coder (M) $; this equality is deduced from \eqref{coco} by applying $\varepsilon \otimes \varepsilon$ to both sides.

The  actions of $\Aut(M)$ and $\coder (M) $ on  $\widetilde A_M^+$ form an action of the Lie pair $(\Aut(M), \coder (M)) $ on  $\widetilde A_M^+$
and induce an action of the Lie pair $(U(M^*), \underline M^*) $ on  $  A_M^+$.
As a consequence,    the  set \begin{equation}\label{Afix+++++} E=(A^+_M)^{{\underline M}^*}= \{ x\in A_M^+\, \vert\, \delta_\varphi( x)=0 \,\, {\rm {for \,\,  all}} \,\, \varphi\in {\underline M}^*\}\end{equation}
is a $U(M^*)$-invariant subalgebra of $A^+_M$.
Since the derivations $\delta_\varphi$ of $A^+_M$ are grading-preserving, ${E}=\oplus_{p\in \ZZ} \, {E}\cap (A^+_M)^p$. This makes ${E}$ into a  commutative graded algebra  with unit   $1_E=e$.

 \subsection{The bracket $\bracket{-}{-}_v^+$ and Hamiltonian reduction}\label{The bracketsa}
 Let $\double{-}{-}$ be  an $n$-graded double bracket  in a unital graded algebra $A$ with $n\in \ZZ$, and let $v$ be a cyclic bilinear form on a counital coalgebra~$M$.   The  induced $n$-graded biderivation $\bracket{-}{-}_v$ in $A_M $  uniquely  extends to a   bracket
 in   $ \kk e\oplus    A_M$ annihilating~$e$ both on the left and on the right. Since $\double{1_A}{A}=\double{A}{1_A}=0$, the latter bracket annihilates $(1_A)_{\alpha}-\varepsilon_M  ( \alpha ) e  $    for all $\alpha\in M$ and descends to an $n$-graded  biderivation $\bracket{-}{-}_v^+$ in   $  A_M^+$.
 The Jacobi form of $\bracket{-}{-}_v^+$ can be computed from Lemma~\ref{bra}.
Hence, if $\double{-}{-}$ is   Gerstenhaber, then so is $\bracket{-}{-}_v^+$.

 Theorem~\ref{justxxxzLie} implies that   the bracket $\bracket{-}{-}_v^+$ is      $\underline M^* $-invariant. As a consequence,    the  graded algebra $E\subset A^+_M$ defined by \eqref{Afix+++++} satisfies   $\bracket{E }{E}_v^+ \subset E$.   This graded algebra   appears in     Hamiltonian reduction as follows.

\begin{theor}\label{mommap-}  Let $B$ be a   graded algebra and  let $p:A\to B$ be a graded algebra epimorphism whose kernel   is generated, as a 2-sided ideal of $A$, by a set $\mathcal M \subset A^{-n}$ satisfying the following condition: for every ${\xi}\in \mathcal M$ there is a scalar $k_{\xi} \in \kk$ such that for all $a\in A$,
$$\double{{\xi}}{a}\equiv k_{\xi}(a\otimes 1_A-1_A\otimes a)\ mod (A\otimes \Ker p + \Ker p \otimes A).$$
Then  the $n$-graded biderivation $\bracket{-}{-}_v^+$ in   $E $   descends uniquely to an $n$-graded biderivation $\bracket{-}{-}$   in the  graded algebra $p^+_M(E)\subset B^+_M$.  If   $\double{-}{-} $ is Gerstenhaber, then so is  $\bracket{-}{-}$.
\end{theor}

\begin{proof}   The uniqueness of  $\bracket{-}{-}$  is obvious, and we need only to prove the existence. We begin by considering   some  ${\xi}\in A^{-n} $ and $k  \in \kk$ such that  for all $a\in A$,
$$\double{{\xi}}{a}= k (a\otimes 1_A-1_A\otimes a). $$
 We claim that for any $\alpha\in M$,   $x\in A^+_M$,
\begin{equation} \label{momo} \bracket{{\xi}_\alpha}{x}^+_v=k \delta_{\overline \alpha}(x)\end{equation}
where the  overbar stands for the cyclic structure   $M\to M^*$ determined by   $v$. Both sides of \eqref{momo} are derivations in $x$ (here we use that ${\xi}\in A^{-n} $). Hence, it is enough to prove \eqref{momo} for each generator $x=a_\beta \in A^+_M$ where $a\in A$, $\beta\in M$.   We have
$$\bracket{{\xi}_\alpha}{a_\beta}^+_v = k \double{{\xi} }{a }^1_{\alpha_\beta} \double{{\xi} }{a }^2_{\beta^\alpha} =k \big ( a_{\alpha_\beta} (1_A)_{\beta^\alpha} -
 (1_A)_{\alpha_\beta} a_{\beta^\alpha}\big )$$
$$=k \big ( a_{\alpha_\beta} \varepsilon ({\beta^\alpha}) -
 \varepsilon ({\alpha_\beta}) a_{\beta^\alpha} \big )=ka_{\varepsilon ({\beta^\alpha}) \alpha_\beta- \varepsilon ({\alpha_\beta})  {\beta^\alpha}}  $$
 where $\varepsilon=\varepsilon_M:M\to \kk$ is the counit of $M$.
  Observe that
 \begin{eqnarray*}\label{momo1}
{\varepsilon ({\beta^\alpha}) \alpha_\beta- \varepsilon ({\alpha_\beta})  {\beta^\alpha}}  &=&  (\id_M \otimes \varepsilon- \varepsilon \otimes \id_M) (\alpha_\beta \otimes \beta^\alpha) \\
\nonumber &=&  (\id_M \otimes \varepsilon- \varepsilon \otimes \id_M) v(\alpha\otimes \beta^2) \beta^1 \otimes \beta^3\\
\nonumber &=&  v(\alpha\otimes \varepsilon (\beta^3) \beta^2) \beta^1 -v(\alpha\otimes \varepsilon (\beta^1) \beta^2) \beta^3\\
&=& v(\alpha\otimes  \beta^2) \beta^1 -v(\alpha\otimes  \beta^1) \beta^2 \\
&=& \overline \alpha( \beta^2) \beta^1 -\overline \alpha( \beta^1) \beta^2 =\delta_{\overline \alpha} (\beta).
\end{eqnarray*}
Therefore  $\bracket{{\xi}_\alpha}{a_\beta}^+_v=ka_{\delta_{\overline \alpha} (\beta)}=k\delta_{\overline \alpha} (a_\beta)$.
 This proves \eqref{momo}.

  The assumptions of the theorem imply  that the kernel,  $J$, of the homomorphism $  p^+_M :A^+_M\to B^+_M$ is generated, as a two-sided ideal of $A^+_M$, by the set $\{{\xi}_\alpha\, \vert \, {\xi}\in \mathcal M, \alpha\in M\}$.  We claim that for all the generators ${\xi}_\alpha $ and all     $x\in A^+_M$,
 \begin{equation} \label{momoweak} \bracket{{\xi}_\alpha}{x}^+_v \equiv k_{\xi} \delta_{\overline \alpha}(x) \quad {\rm mod} \,\,  J . \end{equation}
As above, it suffices to verify \eqref{momoweak} for   $x=a_\beta$ of $A^+_M$ where $a\in A$, $\beta\in M$.  By the assumptions,    $\double{{\xi} }{a }$  is a sum of $k_{\xi}   ( a\otimes 1_A- 1_A\otimes a   )$ and a finite number of vectors    $  b\otimes c \in A\otimes A$
 with $b  $ or $c  $ in $\Ker p$. The contributions of these  terms to $
\bracket{{\xi}_\alpha}{a_\beta}^+_v$ are equal respectively to $
k_{\xi} \delta_{\overline \alpha}(x)$ and $   b_{\alpha_\beta} c_{\beta^\alpha}\in J$.
  This yields \eqref{momoweak}.

  Formula \eqref{momoweak} implies that $\bracket{{\xi}_\alpha}{{E}}^+_v\subset J$ for all ${\xi}\in \mathcal M$, $\alpha \in M$.
 Using the   Leibniz rule in the first variable, we   deduce that $\bracket{J}{{E}}^+_v \subset J$. Therefore,
 $$\bracket{{E} \cap J}{{E}}^+_v \subset \bracket{{E}  }{{E}}^+_v \cap \bracket{  J}{{E}}^+_v \subset  {E} \cap J.$$ By the antisymmetry of
 $\bracket{-}{-}^+_v$,  also
  $\bracket{{E}}{{E} \cap J}^+_v \subset {E} \cap J$.
 Hence, the bracket
 $\bracket{-}{-}^+_v$ in $  {E}$
 descends to a   bracket $\bracket{-}{-}$ in $p^+_M(E)=E/({E} \cap J)$. The properties of $\bracket{-}{-}$ stated in the theorem follow from the properties of
 $\bracket{-}{-}^+_v$.
\end{proof}

   It is clear that in Theorem~\ref{mommap-},
\begin{equation*}  p^+_M(E) \subset (B^+_M)^{{\underline M}^*}= \{ x\in B_M^+\, \vert\, \delta_\varphi(x)=0 \,\, {\rm {for \,\,  all}} \,\, \varphi\in {\underline M}^*\}   \end{equation*}
and   the  map $p^+_M:A^+_M\to B^+_M$  is $U$-equivariant where $U=U(M^*)$. The  action of $U $ on $B^+_M$ restricts to an action of $U $ on $p^+_M(E)$.
Theorem~\ref{justxxxz}   implies that   the bracket $\bracket{-}{-}_v^+$  and   the induced bracket in $p^+_M(E)$ are   $U  $-invariant.

\begin{corol} Let, under the assumptions of Theorem~\ref{mommap-},  $\theta\in M$ be adjoint to $v$ and let $B(\theta)^+$ be the unital subalgebra of $B^+_M$ generated by the set $\{b_\theta\}_{b\in B}$. There is a unique $n$-graded biderivation $\{-,-\}$ in $B(\theta)^+$  such that \begin{equation}\label{nice}  \{    a_\theta,   b_\theta   \}  =     \big ( \double{a}{b}'  \double{a}{b}'' \big )_{ \theta } \end{equation}
for any $a,b\in B$. If    $\double{-}{-} $ is   Gerstenhaber, then so is  $\bracket{-}{-}$.
\end{corol}

\begin{proof} Lemmas~\ref{trat--} and~\ref{trat} imply that  $B(\theta)^+ \subset p^+_M(E)$. Lemma~\ref{trat++} shows that the   bracket in $p^+_M(E)$ provided by Theorem~\ref{mommap-} evaluates on the generators of   $B(\theta)^+$ via \eqref{nice}. Therefore $B(\theta)^+$ is closed  under this bracket. \end{proof}

\subsection{Moment maps}  To apply Theorem~\ref{mommap-}, one may   pick  any  set   $\mathcal M\subset A^{-n}$   and   define   $B$ to be the quotient of $A$ by the two-sided ideal generated by $\mathcal M$. The set $\mathcal M$ may consist of just one element  ${\xi} \in   A^{-n}$. The conditions of Theorem~\ref{mommap-} are satisfied, for example, if
  $\double{{\xi}}{a}=  a   \otimes 1_A- 1_A\otimes  a$ for all $a\in A$.
 In this case, ${\xi}$ is called a {\it moment map}, cf.   \cite{VdB, MT2}. For $n=0$, one   considers   {\it multiplicative moment maps}  $\eta\in   A^0$ such that  for all $a\in A$, $$\double{\eta}{a}= a\otimes  \eta - \eta\otimes a +a \eta \otimes 1_A- 1_A\otimes \eta a , $$
cf.   \cite{VdB, MT2}. Then the  1-element  set $\mathcal M=\{{\xi}\}$ with ${\xi}=\eta-1_A$ satisfies the conditions of Theorem~\ref{mommap-} for $k_{\xi}=2$.

\subsection{Double quasi-Poisson algebras}

   In parallel with  the notion  of a quasi-Poisson bracket  \cite{AKsM},   Van den Bergh     \cite{VdB} introduced a class of double quasi-Poisson     brackets in unital (non-graded) algebras.       A     double bracket $\double{-}{-}$ in      a unital   algebra $A$   is  {\it quasi-Poisson}  if the
induced tribracket  in $A$  is computed by
\begin{eqnarray}\label{compE} \quad \quad
{\triple{a}{b}{c}} &=&  c \otimes a \otimes b + {1_A} \otimes ab \otimes c +
a \otimes {1_A} \otimes bc   +  ca\otimes b \otimes {1_A} \\
\nonumber &&  - a\otimes b \otimes c -  {1_A}  \otimes a \otimes bc -  ca \otimes {1_A} \otimes b -   c \otimes ab \otimes {1_A}
\end{eqnarray}
for any $a,b,c \in A$. The pair $(A,\double{-}{-})$ is   then a  \emph{double quasi-Poisson algebra}.    Note that a double quasi-Poisson
bracket in our sense is $2$ times a double quasi-Poisson bracket in the sense of \cite{VdB}.

\begin{theor}\label{quasi}   Let $(A,\double{-}{-})$ be a double quasi-Poisson algebra,  and let $v$ be a cyclic bilinear form   on a counital coalgebra $M$.   The  restriction of  the  
 bracket $\bracket{-}{-}_v^+$ in $A^{+}_M $ to  the  algebra $E \subset A^{+}_M$ defined by \eqref{Afix+++++}   satisfies the   Jacobi identity and makes ${E} $ into a Poisson algebra.
\end{theor}

\begin{proof} Let $\{-,-,-\}:(A^+_M)^{\otimes 3}\to A^+_M$ be the Jacobi form of    $\bracket{-}{-}_v^+$. We must prove that   $ \{{E}, {E} , {E} \}=0$. We prove a stronger claim  $ \{{E}, A^+_M , A^+_M\}=0$.

  For     $a,b,c\in A$ and   $\alpha, \beta,\gamma\in M$, Lemma~\ref{bra} yields   $ \{a_{\alpha}, b_{\beta} , c_{\gamma}\}=Q-R$ with  $Q$ and $R$ given by \eqref{IMP1}, \eqref{IMP2}. Note that since $A$ is non-graded, the signs in the expressions for $Q$ and $R$ are equal to $+1$. 
 The subindices appearing in $Q$ form a vector in $ M\otimes M\otimes M$ equal to
 \begin{eqnarray*}  \alpha_{(\beta_\gamma)}  \otimes (\beta_\gamma)^\alpha \otimes \gamma^\beta  & {=}& (\widehat v \otimes \id_M) (\id_M\otimes \widehat v)(\alpha\otimes \beta \otimes \gamma)  \\
&=& v(\alpha \otimes \gamma^2) \, v(\beta \otimes \gamma^4)\, \gamma^1 \otimes \gamma^3 \otimes \gamma^5\\
&=& \overline \alpha (\gamma^2) \, \overline \beta ( \gamma^4)\, \gamma^1 \otimes \gamma^3 \otimes \gamma^5  .
 \end{eqnarray*}
Now, each   term  on the right-hand side of  \eqref{compE} contributes a summand to $Q$. The resulting eight summands of $Q$ are  equal to
  \begin{eqnarray*}  x_1   & {=}&    \overline \alpha (\gamma^2) \, \overline \beta ( \gamma^4)\, c_{\gamma^1} a_{\gamma^3} b_{\gamma^5} =
        a_{\overline \alpha (\gamma^2) \gamma^3} b_{\overline \beta ( \gamma^4) \gamma^5} c_{\gamma^1},\\
x_2 & {=}&  \overline \alpha ( \gamma^1) \,\overline \beta ( \gamma^4)\, a_{\gamma^2} b_{\gamma^3} c_{\gamma^5}
=    a_{\overline \alpha ( \gamma^1) \gamma^2} b_{\overline \beta ( \gamma^4) \gamma^3} c_{\gamma^5}, \\
 x_3 & {=}&  \overline \alpha ( \gamma^2) \, \overline \beta ( \gamma^3)\, a_{\gamma^1} b_{\gamma^4} c_{\gamma^5}
 =   a_{\overline \alpha ( \gamma^2) \gamma^1} b_{\overline \beta ( \gamma^3) \gamma^4} c_{\gamma^5},  \\
 x_4  & {=}&  \overline \alpha ( \gamma^3) \, \overline \beta ( \gamma^5)\, c_{\gamma^1} a_{\gamma^2} b_{\gamma^4}
 =     a_{\overline \alpha ( \gamma^3) \gamma^2} b_{\overline \beta ( \gamma^5) \gamma^4} c_{\gamma^1},  \\
   x_5 & {=}&  -  \overline \alpha (\gamma^2) \, \overline \beta ( \gamma^4)\, a_{\gamma^1} b_{\gamma^3} c_{\gamma^5}
   = -    a_{\overline \alpha (\gamma^2)\gamma^1} b_{\overline \beta ( \gamma^4) \gamma^3} c_{\gamma^5},  \\
   x_6  & {=}&  -  \overline \alpha ( \gamma^1) \,\overline \beta ( \gamma^3)\, a_{\gamma^2} b_{\gamma^4} c_{\gamma^5}
   = -     a_{\overline \alpha ( \gamma^1) \gamma^2} b_{\overline \beta ( \gamma^3) \gamma^4} c_{\gamma^5},  \\
    x_7  & {=}&  -  \overline \alpha ( \gamma^3) \, \overline \beta ( \gamma^4)\, c_{\gamma^1} a_{\gamma^2} b_{\gamma^5}
    = -     a_{\overline \alpha ( \gamma^3) \gamma^2} b_{\overline \beta ( \gamma^4) \gamma^5} c_{\gamma^1},  \\
     x_8  & {=}&
        -\overline \alpha ( \gamma^2) \, \overline \beta ( \gamma^5)\, c_{\gamma^1} a_{\gamma^3} b_{\gamma^4}
        =  - a_{\overline \alpha ( \gamma^2) \gamma^3} b_{\overline \beta ( \gamma^5) \gamma^4}  c_{\gamma^1}.
 \end{eqnarray*}
 Here, in the computation of $x_2, x_3$, etc., we use the relation $(1_A)_\alpha= \varepsilon(\alpha) 1_{A^+_M}$.

 Using \eqref{fghje} we obtain that
    \begin{eqnarray*} x_1+x_7  & {=}&  a_{\overline \alpha ( \gamma^2) \gamma^3-\overline \alpha (\gamma^3) \gamma^2 } b_{\overline \beta ( \gamma^4) \gamma^5} c_{\gamma^1}  \\
&=&  a_{-\delta_{\overline \alpha}   (\gamma^2)  } b_{\overline \beta ( \gamma^3) \gamma^4} c_{\gamma^1} \\
&=&  a_{ \delta_{\overline {\gamma^2}}   (\alpha)  } b_{\overline \beta ( \gamma^3) \gamma^4} c_{\gamma^1}= \delta_{\overline {\gamma^2}} (a_\alpha) b_{\overline \beta ( \gamma^3) \gamma^4} c_{\gamma^1}.
 \end{eqnarray*}
 Similarly, $$x_2+x_5= \delta_{\overline {\gamma^1}} (a_\alpha) b_{\overline \beta ( \gamma^3) \gamma^2} c_{\gamma^4},$$
 $$x_3+x_6= -\delta_{ \overline {\gamma^1}} (a_\alpha) b_{\overline \beta ( \gamma^2) \gamma^3} c_{\gamma^4},$$
 $$x_4+x_8= -\delta_{ \overline {\gamma^2}} (a_\alpha) b_{\overline \beta ( \gamma^4) \gamma^3} c_{\gamma^1}.$$
 Analogous computations  allow us to deduce further  that
 $$ x_1+x_7 + x_4+x_8=  \delta_{\overline {\gamma^2}} (a_\alpha)  \delta_{\overline {\gamma^3}} (b_\beta  ) c_{\gamma^1},$$
 $$x_2+x_5+x_3+x_6= -\delta_{\overline {\gamma^1}} (a_\alpha)  \delta_{\overline {\gamma^2}} (b_\beta  ) c_{\gamma^3}.$$
Summing up, we obtain that \begin{equation}\label{newQ} Q=Q(a,b,c, \alpha, \beta, \gamma)= \delta_{\overline {\gamma^2}} (a_\alpha)  \delta_{\overline {\gamma^3}} (b_\beta  ) c_{\gamma^1}  -\delta_{\overline {\gamma^1}} (a_\alpha)  \delta_{\overline {\gamma^2}} (b_\beta  ) c_{\gamma^3}. \end{equation}
Similarly, starting from the expression
 \begin{eqnarray*}  \alpha_{(\gamma^\beta)}  \otimes (\gamma^\beta)^\alpha \otimes \beta_\gamma  & {=}& (\widehat v \otimes \id_M) p_{23} (\id_M\otimes \widehat v)(\alpha\otimes \beta \otimes \gamma)  \\
&=& v(\alpha \otimes \gamma^4) \, v(\beta \otimes \gamma^2)\, \gamma^3 \otimes \gamma^5 \otimes \gamma^1\\
&=& \overline \alpha (\gamma^4) \, \overline \beta ( \gamma^2)\,  \gamma^3 \otimes \gamma^5 \otimes \gamma^1
 \end{eqnarray*}
 we obtain that
  \begin{equation}\label{newR} R=R(a,b,c, \alpha, \beta, \gamma)=  \delta_{\overline {\gamma^3}} (a_\alpha)  \delta_{\overline {\gamma^2}} (b_\beta  ) c_{\gamma^1}  -\delta_{\overline {\gamma^2}} (a_\alpha)  \delta_{\overline {\gamma^1}} (b_\beta  ) c_{\gamma^3}. \end{equation}

 These expansions  for $Q$ and $R$ yield an expansion of  $ \{a_{\alpha}, b_{\beta} , c_{\gamma}\}=Q-R$. Now,
using the fact  that the Jacobi form is a derivation in the first and second variables (cf. \eqref{jjja}) we deduce that for any $x,y\in A^+_M$,
$$ \{x, y , c_{\gamma}\}= \big  ( \delta_{\overline {\gamma^2}} (x)  \delta_{\overline {\gamma^3}} (y  )   -\delta_{\overline {\gamma^3}} (x)  \delta_{\overline {\gamma^2}} (y  ) \big)  c_{\gamma^1} + \big ( \delta_{\overline {\gamma^2}} (x)  \delta_{\overline {\gamma^1}} (y  ) -\delta_{\overline {\gamma^1}} (x)  \delta_{\overline {\gamma^2}} (y  ) \big ) c_{\gamma^3}  .$$
If $x\in {E} $, then     the inner derivations annihilate $x$, and so  $ \{x, A^+_M , c_{\gamma}\}=0$ for   all
the generators $c_\gamma$ of $A^+_M$. Since the Jacobi form is a derivation in the third variable, we deduce that
$ \{x, A^+_M , A^+_M\}=0$. Thus, $ \{{E}, A^+_M , A^+_M\}=0$.
\end{proof}

\begin{corol}\label{quasi+} If,  under the assumptions of Theorem~\ref{mommap-},  $n=0$,  $A$ is concentrated in   degree zero,  and the double bracket in $A$ is quasi-Poisson, then the induced  bracket in $p_M^+(E)$   is a Poisson bracket.   \end{corol}

 Theorem \ref{quasi} and Corollary~\ref{quasi+} are due to   Van den Bergh  \cite{VdB} for $M $    and $v$ as in Example \ref{exam}.4.
 
 \subsection{Group algebras}\label{Group algebras} We discuss in more detail the case where $A=\kk[\pi]$ is the (non-graded)  algebra  of a group~$\pi$. 
 Consider a counital coalgebra $M$ whose underlying module is free of finite rank and consider  the dual algebra $M^*$.  For any  (non-graded) algebra $B$, we have $H_M(B)=B\otimes M^*$. If $B$ is unital, then so is the    algebra  $ B \otimes M^*$, and we can identify the  set $\Hom_{\grAlg^+} (A,  H_M(B))$   with the set   of group homomorphisms  from  $\pi  $ to    the   group  $ U(B\otimes M^*)$  of  invertible elements of    $ B \otimes M^*$.   The bijection \eqref{uncogral} exhibits   $A^+_M$ as the coordinate algebra of the functor $B\mapsto  \Hom (\pi, U(B\otimes M^*))$ where $B$ runs over unital commutative (non-graded) algebras. For example, if  $M=(\Mat_N({{{{\mathcal A}}}}))^*$ where $N\geq 1$ and ${{{{\mathcal A}}}}$ is a symmetric Frobenius algebra, then $U(B\otimes M^*)= GL_N(B \otimes {{{{\mathcal A}}}})$  is the group of  invertible elements of the matrix algebra $B\otimes \Mat_N({{{{\mathcal A}}}})=\Mat_N(B \otimes {{{{\mathcal A}}}})$. If  $M=(\kk[G])^*$ where $G$ is a finite group, then $U(B\otimes M^*) $  is the group of  invertible elements of the group algebra $B\otimes \kk[G] =B[G]$.
 
  A   double   bracket $\double{-}{-}$  in $A$ and a  cyclic bilinear form $v$ on $M$ induce    a (zero-graded) biderivation $\bracket{-}{-}_v^+$ in $A^+_M$.  If $\double{-}{-}$ is  Poisson, then so is $\bracket{-}{-}_v^+$. If $\double{-}{-}$ is   quasi-Poisson, then the restriction of $\bracket{-}{-}_v^+$ to  the  algebra $E \subset A^{+}_M$ defined by \eqref{Afix+++++}  is  a Poisson bracket.

 These constructions apply   to the fundamental group  $\pi=\pi_1(\Sigma, \star)$  of an oriented surface $\Sigma$ with   boundary and with base point  $\star \in \partial \Sigma$. The algebra $A=\kk[\pi]$ carries a canonical double quasi-Poisson bracket, see \cite{MT2}. This induces   a   biderivation $\bracket{-}{-}_v^+$ in $A^+_M$ which  restricts to   a Poisson bracket in $E \subset A^{+}_M$.
 When $M=(\Mat_N(\RR))^*$ and   $\Sigma$ is compact, one can identify $A^{+}_M$ and $E$ with the algebras of regular functions on   $\Hom(\pi, GL_N(\RR)) $ and    $\Hom(\pi, GL_N(\RR))/GL_N(\RR)$, respectively. In this case, the bracket $\bracket{-}{-}_v^+$ in $E$ is well-known; its extension to $A^{+}_M$ was first constructed   in \cite{AKsM}  up to isomorphism.
 For more on this, see   \cite{MT2}.
 
 The Hamiltonian reduction has the following interpretation in the context of surfaces. Assume that   $\partial \Sigma \cong S^1$. Attaching a 2-disk to $\Sigma$, we obtain a  surface $\Sigma'$. The inclusion $\Sigma \hookrightarrow \Sigma'$ induces an   epimorphism $A=\kk[\pi] \to B=\kk[\pi_1(\Sigma', \star)]$ with kernel   $A(\eta-1)A$ where $\eta\in \pi$ is represented by   $\partial \Sigma$. The element  $\eta$ is a multiplicative moment map.  Corollary~\ref{quasi+} implies that the Poisson bracket $\bracket{-}{-}_v^+$ in $E$ descends to a Poisson bracket in   $p_M^+(E) \subset B^+_M$ whih is again well-known when $M=(\Mat_N(\RR))^*$ and   $\Sigma$ is compact.

\section{Equivariant representation algebras and brackets}

We introduce   equivariant representation algebras and construct  brackets in them.

\subsection{Algebras over enriched groups} An {\it enriched group} $G$ is a group endowed with a distinguished subgroup $G_0$ of index 1 or 2. 
 Given an enriched group $G$, we introduce graded $G$-algebras and $G$-coalgebras as follows.  A {\it graded $G$-algebra}  is a graded algebra  $A$  equipped
 with a  left  action of $G$   such that the elements of $G_0$ act by graded algebra automorphisms of $A$ 
 and the elements of $G\backslash G_0$ act by graded algebra antiautomorphisms of $A$. Here   a {\it graded algebra antiautomorphism} of   $A$ is a grading-preserving linear isomorphism $f:A\to A$ such that $f(ab)= (-1)^{\vert a\vert \vert b\vert} f(  b) f( a)$
 for all homogeneous $a,b\in A$.
 

A   map  between graded $G$-algebras  is {\it equivariant} if it commutes with the action of $G$. Graded $G$-algebras and equivariant graded algebra homomorphisms form a category  denoted $G-\grAlg$.

A {\it $G$-coalgebra}  is a   coalgebra $M$ with comultiplication ${{\mu}}$ 
 equipped
 with a  left  action of $G$   such that the elements of $G_0$ act by coalgebra automorphisms of $M$ 
 and the elements of $G\backslash G_0$ act by coalgebra antiautomorphisms  of $M$. Here   an {\it  antiautomorphism} of   $M$ is a linear isomorphism $h:M\to M$ such that $(h \otimes h) {{\mu}} =p_{21} {{\mu}} h$  where    $p_{21}:M\otimes M\to M\otimes M$ is the permutation of the tensor factors.

 Given a $G$-coalgebra $M$ and a  graded algebra $B$, we  provide the convolution algebra $H_M(B)=\Hom_\kk(M,B)$ with the   left action of $G$ by $$gf= f\circ (g^{-1}:M\to M)$$  for   $g\in G$, $f\in H_M(B)$.
 All $g\in G_0 \subset G$ act on $H_M(B)$ by algebra automorphisms:  
 $$g(f_1 f_2)   = (f_1 f_2) \circ  g^{-1} = {m_B} (f_1 \otimes f_2) {{\mu}}  \circ  g^{-1}
   = {m_B} (f_1 \otimes f_2) \circ (g^{-1} \otimes g^{-1}) {{\mu}}  $$
   $$ = {m_B} ((f_1\circ  g^{-1}) \otimes (f_2 \circ g^{-1}) ) {{\mu}}   
 = {m_B} (g f_1   \otimes g f_2  ) {{\mu}}   = (g f_1) ( g f_2  )  $$ 
 where $f_1, f_2\in H_M(B)$ and  ${m_B}:B\otimes B\to B$ is the multiplication in $B$. If    $B$ is   commutative, then all $g\in G\backslash G_0  $ act on $H_M(B)$ by algebra antiautomorphisms:   for  homogeneous $f_1, f_2\in H_M(B)$,
 $$g(f_1 f_2)    = {m_B} (f_1 \otimes f_2) {{\mu}}  \circ  g^{-1}
   = {m_B} (f_1 \otimes f_2) \circ p_{21} (g^{-1} \otimes g^{-1})   {{\mu}}  $$
   $$  = (-1)^{\vert f_1\vert \vert f_2 \vert} \, {m_B} (f_2 \otimes f_1) \circ   (g^{-1} \otimes g^{-1})   {{\mu}}    = (-1)^{\vert f_1\vert \vert f_2 \vert} (g f_2) ( g f_1  ) . $$
 Thus,  in this case,   $H_M(B)$ is a graded $G$-algebra.

\subsection{Equivariant representation algebras}\label{AMT0}
Fix an enriched group $G$, a graded $G$-algebra $A$, and a $G$-coalgebra $M$. Let      ${  A}_M^G$ be the commutative graded algebra obtained as the quotient of $ A_M$ by   the    relations
$\{ (ga )_{g\alpha }=a_{ \alpha }  \, \vert \, 
 a \in A, \alpha \in {M}\}$. Note the   equivalent system of relations 
   $\{ (ga)_{\alpha } =   a_{g^{-1}\alpha }  \, \vert \,
 a \in A, \alpha \in {M}\}$.
  
 We call $A^G_M$ the {\it equivariant representation algebra of $A$  with respect to $M$.}  
  The construction of $A_M^G$ is functorial: a  morphism $f:A\to B$ in the category $G-\grAlg$ induces a
 graded  algebra homomorphism $ f_M^G: A_M^G \to B_M^G$ by
$  f_M^G (a_{\alpha})  = (f(a))_{\alpha}$ for all $a\in A$, $\alpha \in {M}$.

\begin{lemma}\label{justequi}  
For any commutative graded algebra $B$, there is a canonical bijection
\begin{equation}\label{eq:adjunctionequi}
\Hom_{ \grCom } (   A^G_{M} , B)
\stackrel{\simeq}{\longrightarrow} \Hom_{G-\grAlg} (A,  H_M(B))
\end{equation}
which is natural in $A$ and $B$.
\end{lemma}

\begin{proof}
The map \eqref{eq:adjunctionequi}  carries
a   graded algebra homomorphism $r:  A_{M}^G\to B$    to the
linear  map   $ s=s_r: A \to H_M(B)$  defined by
  $s(a)(\alpha)= r( a_{\alpha }) $ for all $a\in A$ and $\alpha\in M$. The    proof of Lemma~\ref{just} shows that 
    $s$ is   a graded algebra homomorphism. It is equivariant: $s(ga)= g (s(a))$ for any $g\in G$ and $a\in A$. Indeed, for all  $\alpha\in M$,
    $$ s(ga) (\alpha) = r( (ga)_{\alpha })= r( a_{g^{-1}\alpha })= s(a) (g^{-1} \alpha)=  g (s(a)) (\alpha).$$

The map inverse to \eqref{eq:adjunctionequi}  carries
     an equivariant graded algebra homomorphism  $s:A \to H_M(B)$ to the
algebra homomorphism   $ r=r_s:  A_{M}^G\to B$ defined on the   generators    by $r( a_{\alpha })=s(a)(\alpha)$.       This rule  is compatible with   the  defining relations in $  A_M$, see the  proof of Lemma~\ref{just}, and with the  relations $(ga )_{g\alpha }=a_{ \alpha } $:
$$r(  (ga)_{g\alpha})  = s(ga) (g\alpha)= (g s(a)) (g\alpha)= s(a)( g^{-1}g\alpha)=s(a)(\alpha) =r( a_{\alpha }).$$
 Clearly,    the maps $s\mapsto r_s$ and $r\mapsto s_r$ are mutually inverse.
 \end{proof}

  Lemma~\ref{just} yields the following universal  property of $  A_M^G$: for any commutative graded  algebra $B$ and any equivariant graded  algebra   homomorphism  $s:A\to H_M(B)$, there is a unique graded  algebra  homomorphism  $r:   A_M^G\to B$ such that the  following diagram
 commutes:  $$
\xymatrix{ A  \ar[r]   \ar[rd]_{s} & H_M(   A_M^G)\ar[d]^{H_M(r)}  \\
&   H_M(B).
}
$$   Here the horizontal arrow is the equivariant graded  algebra   homomorphism carrying  $a\in A$ to the   map $M\to    A_M^G, \alpha\mapsto a_\alpha$.


\subsection{A bracket  in $A^G_M$}  Let $G, A, M$ be as in Section \ref{AMT0}. Under certain assumptions on a double bracket $\double{-}{-}$  in $A$ and a cyclic bilinear form $v$ on $M$, we   derive from them a bracket in $A^G_M$. We say that   $\double{-}{-}$    is   {\it equivariant} if
for all $g\in G$, $a,b \in A$, 
 $$
\double{ga}{gb } =
\left\{\begin{array}{ll} (g \otimes g)   (\double{a}{b})  & \hbox{if } g\in G_0, \\ P_{21} (g \otimes g)   (\double{a}{b}) & \hbox{if } g \in G\backslash G_0    \end{array}\right.
$$
where $P_{21}$ is the graded permutation in $A\otimes A$ defined in Section \ref{bibr}.
We say that $v$   is   {\it equivariant} if $v(g\alpha \otimes g\beta )=v(\alpha \otimes \beta)$ for all $g\in G$, $\alpha, \beta \in  M $.

 \begin{lemma}\label{brainv}
    Let        $  \bracket{-}{-}_v   $ be the $n$-graded biderivation in $A_M$   derived from  an equivariant $n$-graded  double bracket $\double{-}{-}$  in $A$ and  an equivariant cyclic  bilinear form $v$ on $M$. Let $q:A_M\to A^G_M$ be the natural projection. If $G$ is finite, then there is a unique $n$-graded biderivation    $ \bracket{-}{-}^G_v   $ in $  A_M^G$  such that for any
  $a,b\in   A $, $\alpha, \beta \in M$,
\begin{equation}\label{mainfinv}\bracket{q(a_\alpha)}{q(b_\beta)}^G_v= \sum_{g\in G}  \, q(\{(ga)_{g\alpha} , b_\beta\}_v) . \end{equation}
 If $\double{-}{-}$ is a double Gerstenhaber bracket, then  $ \bracket{-}{-}^G_v   $ is a Gerstenhaber bracket.
\end{lemma}

\begin{proof} The uniqueness of the bracket $ \bracket{-}{-}^G_v$ is obvious, and we need only to prove the existence.  We begin with   a few computations in the  $G$-coalgebra $M$. First of all,   for  any  $g\in G_0$ and   $\beta\in M$,
$$ (g\beta)^1 \otimes (g\beta)^2 \otimes (g\beta)^3= {{\mu}}^2(g\beta)= g(\beta^1) \otimes g(\beta^2) \otimes g(\beta^3) $$ 
Using the   equivariance of $v$, we deduce   that  for any $\alpha, \beta \in M$,
 $$\widehat v (g\alpha \otimes g\beta) =v(g\alpha \otimes (g\beta)^2) (g\beta)^1 \otimes (g\beta)^3=
 v(g\alpha \otimes g(\beta^2))  g(\beta ^1) \otimes  g(\beta^3)$$
 $$= v( \alpha \otimes  \beta^2) (g \otimes g) (\beta^1  \otimes   \beta^3)=
(g \otimes g) \widehat v (\alpha \otimes \beta).$$
 We   rewrite this in the notation of Section~\ref{mapfv}: 
\begin{equation}\label{vvd} (g\alpha)_{(g\beta )} \otimes  (g \beta )^{g(\alpha)}= g(\alpha_\beta) \otimes g(\beta^\alpha).\end{equation}

Similarly, for     $g\in G\backslash G_0$ and $\alpha, \beta \in M$,  
$$ (g\beta)^1 \otimes (g\beta)^2 \otimes (g\beta)^3= {{\mu}}^2(g\beta)= g(\beta^3) \otimes g(\beta^2) \otimes g(\beta^1) $$ 
and therefore
$$\widehat v (g\alpha \otimes g\beta) =v(g\alpha \otimes (g\beta)^2) (g\beta)^1 \otimes (g\beta)^3=
 v(g\alpha \otimes g(\beta^2))  g(\beta^3) \otimes  g(\beta^1)$$
 $$= v( \alpha \otimes  \beta^2) p_{21} (g \otimes g) (\beta^1  \otimes   \beta^3)=
p_{21} (g \otimes g) \widehat v (\alpha \otimes \beta).$$
In the notation of Section~\ref{mapfv} this gives
\begin{equation}\label{vvd+} (g\alpha)_{(g\beta )} \otimes  (g \beta )^{g(\alpha)}=  g(\beta^\alpha) \otimes   g(\alpha_\beta).\end{equation}

Next, we define a left action of $G$ on $A_M$ by graded algebra automorphisms. 
Each  $g\in G$ acts on the generators by $g \, a_\alpha  = (ga)_{g\alpha}$ for   $a\in A$, $\alpha \in M$. We   check the compatibility with the defining relations of $A_M$. The compatibility  with the bilinearity relations is obvious. Consider  the multiplicativity relation  $(ab)_\alpha=a_{\alpha^1 } b_{\alpha^2}$   with homogeneous $a,b\in A$ and $\alpha  \in M$. If $g\in G_0$, then
$$g (ab)_\alpha= (g(ab))_{g\alpha}=(g(a) g(b))_{g\alpha}=(ga)_{g\alpha^1} (gb)_{g\alpha^2}= g(a_{\alpha^1}) g(b_{\alpha^2})= g(a_{\alpha^1}  b_{\alpha^2} ). $$
If $g\in G \backslash G_0$, then  
$$g (ab)_\alpha= (g(ab))_{g\alpha}= (-1)^{\vert a\vert \vert b\vert} (g(b) g(a))_{g\alpha}   = (-1)^{\vert a\vert \vert b\vert}  (gb)_{(g\alpha)^1} (ga)_{(g\alpha)^2}$$
$$=(ga)_{(g\alpha)^2} (gb)_{(g\alpha)^1} =(ga)_{g\alpha^1} (gb)_{g\alpha^2} =    g(a_{\alpha^1})\, g(  b_{\alpha^2} ) = g(a_{\alpha^1}   b_{\alpha^2} ).$$
The compatibility with the commutativity relation  $a_{\alpha } b_{\beta}=(-1)^{\vert a\vert \vert b\vert} b_{\beta } a_{\alpha}$: for all $g $,  $$ g(a_{\alpha } b_{\beta})=  g(a_{\alpha }) g( b_{\beta})=(ga)_{g\alpha } (gb)_{g\beta}= (-1)^{\vert a\vert \vert b\vert} (gb)_{g\beta } (ga)_{g\alpha}= (-1)^{\vert a\vert \vert b\vert} g(b_{\beta } a_{\alpha}).$$

The bracket $\bracket{-}{-}=\bracket{-}{-}_v  $ in $A_M$ is invariant under the action of $G$. Indeed, it is enough to check that $\bracket{g(a_\alpha)}{g(b_\beta)}=g\bracket{a_\alpha}{b_\beta}$ for any $g\in G$, 
  $a,b\in A$,  $\alpha, \beta  \in M$. We have
\begin{equation}\label{vvd++}\bracket{g(a_\alpha)}{g(b_\beta)}=\bracket{(ga)_{g\alpha} }{(gb)_{g\beta}} =\double{ga}{gb}'_{(g\alpha)_{(g\beta )}}  \double{ga}{gb}''_{ (g \beta )^{g\alpha}}.\end{equation}
Set  $x=x'\otimes x''=\double{a}{b}$.  If $g\in G_0$, then   the equivariance of $\double{-}{-}$ and \eqref{vvd} imply that  the right-hand side of \eqref{vvd++} is equal to 
 $(gx')_{g(\alpha_\beta ) }  (gx'')_{ g (\beta^\alpha)}$. 
If $g\in G \backslash G_0$, then   using  \eqref{vvd+} we compute the right-hand side of \eqref{vvd++} to be 
$$(-1)^{\vert x'\vert \vert x''\vert} (gx'')_{g(\beta^\alpha ) }  (gx')_{ g ( \alpha_\beta)}= (gx')_{ g ( \alpha_\beta)} (gx'')_{g(\beta^\alpha ) } .$$
In both cases, $$\bracket{g(a_\alpha)}{g(b_\beta)}=(gx')_{g(\alpha_\beta ) }  (gx'')_{ g (\beta^\alpha)}=g (x'_{ \alpha_\beta  } ) g( x''_{\beta^\alpha }) =g (x'_{ \alpha_\beta  }  x''_{\beta^\alpha })= g(\bracket{a_\alpha}{b_\beta}).$$

It is obvious   that $q (gz)=q(z)$ for all $g\in G$ and $z\in A_M$. The $G$-invariance of   the bracket $\bracket{-}{-}$ in $A_M$ implies that for any $x,y\in A_M$,
\begin{equation*}\label{qg} q(\{g x , y\})= q(g\{  x , g^{-1} y\})= q( \{  x , g^{-1} y\}). \end{equation*}
  Summing up over all $g\in G$, we obtain that 
$$ \sum_{g\in G}  \, q(\{g x , y\})= \sum_{g\in G}  \, q( \{  x , g  y\}) \in A^G_M .$$
Denote this sum  by $[x,y]$.
This defines a bilinear pairing $[-,-]:A_M\times A_M\to A^G_M$. The $n$-graded antisymmetry for $\bracket{-}{-}$ implies that
for   homogeneous $x,y \in A_M$,
$$[x,y]= \sum_{g\in G}  \, q(\{g x , y\}) =- (-1)^{\vert x\vert_n  \vert y\vert_n} \sum_{g\in G}  \, q(\{y, g x  \})=- (-1)^{\vert x\vert_n  \vert y\vert_n} [y,x].$$
The first Leibniz rule  for $\bracket{-}{-}$ implies that for any homogeneous $x,y,z\in A_M$,
$$[x,yz]=[x,y] q(z)+ (-1)^{\vert x\vert_n \vert y\vert} q(y) [x,z].$$
 Observe also that for any $h\in G$, $a\in A$, $\alpha\in M$, 
 $$[x, (ha)_{h\alpha}-a_\alpha]=  
 \sum_{g\in G}  \, q( \{  x ,  (gha)_{gh\alpha}  \}) -  \sum_{g\in G}  \, q( \{  x , (ga)_{g\alpha} \})=0.$$
 These two formulas imply that $[A_M, \Ker q]=0$. The antisymmetry of $[-,-]$ yields 
 $[\Ker q, A_M]=0$. Hence,  the pairing $[-,-]:A_M\times A_M\to A^G_M$ descends to an $n$-graded biderivation $ \bracket{-}{-}^G_v$ in $ A^G_M$ satisfying the conditions of the lemma.
 
 The  Jacobi form $\{-,-,-\}$ of $ \bracket{-}{-}^G_v$    can be  computed on the generators   using Lemma~\ref{bra}.    For   homogeneous $a,b,c\in A$ and any $\alpha, \beta, \gamma\in M$, we obtain 
  \begin{equation}\label{IMP1-ee}  \{q(a_{\alpha}), q(b_{\beta}) , q(c_{\gamma})\} \end{equation} 
   \begin{equation*} =  \sum_{g,h\in G}  q\big ( Q (ga,hb,c, g\alpha, h\beta, \gamma)  -R  (ga,hb,c, g\alpha, h\beta, \gamma)  \big )\end{equation*} where
$Q, R\in A_M$ are defined by \eqref{IMP1}, \eqref{IMP2}
Therefore, if $\double{-}{-}$ is   Gerstenhaber, then so is $ \bracket{-}{-}^G_v   $.
\end{proof}

\subsection{The unital case}\label{uniuni} Suppose  now that the graded $G$-algebra $A$ is unital  and the $G$-coalgebra $M$ is counital. It is understood that the action  of   $G$ on $A$   fixes  the unit $1_A$ and  the action  of   $G$ on   $M$ fixes  the counit $\varepsilon=\varepsilon_M$.
 We construct  a unital graded algebra $A_M^{G+}$ by  adjoining a two-sided unit $e$ to $A_M^G$ and quotienting the resulting algebra $\kk e \oplus A^G_M$ by the   relations  $\{(1_A)_{\alpha}=\varepsilon   ( \alpha ) e\}_{\alpha\in M}$.
 For any unital commutative graded algebra $B$, the bijection  \eqref{eq:adjunctionequi} induces  a natural bijection 
\begin{equation}\label{eq:adjunctionequinnn}
\Hom_{ \grCom^+ } (   A^{G+}_{M} , B)
\stackrel{\simeq}{\longrightarrow} \Hom_{G-\grAlg^+} (A,  H_M(B))
\end{equation}
 where    $\grCom^+$ is the category   of unital
commutative  graded algebras and  $G-\grAlg^+$ is the category of unital $G$-algebras and equivariant graded algebra homomorphisms carrying $1$ to $1$. The construction of $  A^{G+}_{M}$    obviously extends to a   functor $f\mapsto  f^{G+}_{M}:  G-\grAlg^+ \to {\grCom}^+$.

  We introduce an equivariant analogue $\mathcal E\subset A^{G+}_{M}$ of the algebra  \eqref{Afix+++++}. To this end, we 
 define a right action of $G$ on  $M^*$ by 
$(\varphi g)(\alpha)=\varphi (g \alpha)$ for   $\varphi\in M^*, g\in G , \alpha \in M$. Let ${{L}}$   be the set  of all $\varphi\in M^*$ such that 
$ \varphi g =\varphi  $ for all $g\in G_0 $ and 
$ \varphi g = - \varphi  $ for all $g\in G\backslash G_0 $. It is easy to see that  ${{L}}$ is a Lie subalgebra of the Lie algebra $\underline M^{* }$ and the 
  coderivations  $\{\delta_\varphi:M\to M \}_{\varphi\in {{L}} }$  are $G$-equivariant. The induced derivations  $\{\delta_\varphi\}_{\varphi\in {{L}} }$ of     $  A_M $ commute  with the action of $G$ on   $  A_M $ defined in the proof of Lemma \ref{brainv} and  induce  an action of ${{L}}$ on 
$  A_M^{G+}$  by derivations. Then \begin{equation*}\label{Afix+++++more} \mathcal E = \{ x\in A_M^{G+} \, \vert\, \delta_\varphi( x)=0 \,\, {\rm {for \,\,  all}} \,\, \varphi\in {{L}} \}\end{equation*}
is a   unital   graded  subalgebra of $A_M^{G+}$.

Suppose now that $G$ is finite, $A$ carries   an equivariant $n$-graded double bracket,   and $M$ carries an equivariant cyclic bilinear form $v$. The   bracket $\bracket{-}{-}^G_v$ in $A_M^G $      extends   to a   bracket
 in   $ \kk e\oplus    A_M$ annihilating   $e$   on the left and on the right. The latter bracket annihilates $(1_A)_{\alpha}-\varepsilon   ( \alpha ) e  $    for all $\alpha\in M$ and descends to an $n$-graded  biderivation $\bracket{-}{-}^{G+}_v$ in   $  A_M^{G+}$ invariant under the action of the Lie algebra $L$.  By definition, for any
  $a,b\in   A $, $\alpha, \beta \in M$,
\begin{equation*}\label{mainfinvE} \bracket{q^+(a_\alpha)}{q^+(b_\beta)}^{G+}_v= \sum_{g\in G}  \, q^+(\{(ga)_{g\alpha} , b_\beta\}_v) = \sum_{g\in G}  \, q^+(\{a_{\alpha} , (gb)_{g\beta}\}_v)   \end{equation*}
where $q^+$ is the natural projection $A_M\to  A_M^{G+}$.
 Moreover,  $\bracket{\mathcal E}{\mathcal E}^{G+}_v \subset \mathcal E$.
If $\double{-}{-}$ is   Gerstenhaber, then so is $\bracket{-}{-}^{G+}_v$. 

We   state equivariant versions of Theorems~\ref{mommap-} and \ref{quasi}.

\begin{theor}\label{mommap-E}  {\rm (Equivariant  Hamiltonian reduction)} Under the assumptions above, let  $B$ be a graded $G$-algebra  and let $p:A\to B$ be an equivariant algebra epimorphism whose kernel satisfies the same condition as in   Theorem~\ref{mommap-}. Then the bracket $\bracket{-}{-}^{G+}_v$ in ${\mathcal E}$ descends to an $n$-graded Leibniz bracket   in the graded algebra $p^{G+}_{M}({\mathcal E})\subset B^{G+}_{M}$.    If   $\double{-}{-} $ is Gerstenhaber, then so is  the latter bracket.
\end{theor}

\begin{proof}    
We define a map $\ell:  M^* \to L \subset    M^*$ by
  $${{\ell(\varphi)}}= \sum_{g\in G_0} { {\varphi g}  }-  \sum_{g\in G\backslash G_0}  { {\varphi g}   }. $$
  We claim that for any $\varphi\in M^*, g\in G , b\in A, \beta\in M$
\begin{equation} \label{elll} \sum_{g\in G}  \,   q^+ \delta_{\varphi} ( (gb)_{g\beta} ) =  q^+ \delta_{\ell({\varphi })} (b_\beta) .\end{equation} Indeed, it follows from the definitions that
 \begin{equation*} \label{momoE1} \delta_\varphi (g\beta)= \begin{cases}
g \delta_{\varphi g} (\beta)  & {\rm if}  \,\,  g\in G_0 \\
- g \delta_{\varphi g} (\beta)  & {\rm if}  \,\, g\in G\backslash G_0. 
\end{cases} \end{equation*}
Therefore
$$ \sum_{g\in G}  \,   q^+ \delta_{\varphi} ( (gb)_{g\beta} ) = 
 q^+( \sum_{g\in G}  (gb)_{\delta_{\varphi} (g\beta)})=
  q^+(\sum_{g\in G_0} (gb)_{g\delta_{{\varphi} g}  (\beta)}-  \sum_{g\in G\backslash G_0} (gb)_{g\delta_{{\varphi} g}  (\beta)}) $$
$$=
  q^+(\sum_{g\in G_0} b_{\delta_{{\varphi} g}  (\beta)}-  \sum_{g\in G\backslash G_0} b_{\delta_{{\varphi} g}  (\beta)}) =  q^+(b_{\delta_{{\ell({\varphi})}}(\beta)})   =  q^+ \delta_{\ell({\varphi })} (b_\beta)$$
    The rest of the argument
     follows the proof of Theorem~\ref{mommap-} with   \eqref{momo} replaced by  
 \begin{equation*} \label{momoE} \bracket{q^+({\xi}_\alpha)}{x}^{G+}_v= k q^+ \delta_{\ell({\overline  \alpha })} (x)  \end{equation*}
for any $x\in A_M^{G+}$.  We check this equality for $x=q^+(b_\beta)$ with $b\in A, \beta\in M$:
\begin{equation*}\bracket{q^+({\xi}_\alpha)}{q^+(b_\beta)}^{G+}_v=  \sum_{g\in G}  \, q^+(\{{\xi}_{\alpha} , (gb)_{g\beta}\}_v) \end{equation*} \begin{equation*} =   \sum_{g\in G}  \, k q^+ \delta_{\overline  \alpha} ( (gb)_{g\beta} )     =k q^+ \delta_{\ell({\overline  \alpha })} (b_\beta). \qedhere \end{equation*}
 \end{proof}   
  
 \begin{theor}\label{quasiE}   If, under the assumptions above, $A$ is concentratred in degree zero, $n=0$, and the double bracket in $ A $ is quasi-Poisson, then the  restriction of  the  
 bracket $\bracket{-}{-}_v^{G+} $   to   $\mathcal E \subset A_M^{G+}$    satisfies the   Jacobi identity and makes ${\mathcal E} $ into a Poisson algebra.
\end{theor}

\begin{proof} 
To compute the Jacobi form of the bracket $\bracket{-}{-}_v^{G+} $  we use the expansion \eqref{IMP1-ee} with $q$ replaced by $q^+$.  To compute the terms of this expansion  we use   \eqref{newQ}, \eqref{newR} and to compute their sum over all $g,h\in G$ we use  \eqref{elll}. 
The rest of the arguments  is as in the proof of Theorem~\ref{quasi}.
  \end{proof}
 
\subsection{The involutive case}  We focus now on the case  where $G$ is a cyclic group of order two  with generator $\iota$ and   $G_0\subset G$ is the trivial subgroup. To turn a unital graded algebra $A$ into a $G$-algebra, one needs only to fix an involutive graded antiautomorphism  $\iota$  of $A$ such that $\iota(1_A)=1_A$.  An  $n$-graded  double bracket $\double{-}{-}$  in $A$ is equivariant if and only if for all   $a,b \in A$,
\begin{equation}\label{iotaio}
\double{\iota(a) }{ \iota(b)  } =
 P_{21} (\iota \otimes \iota)   ( \double{a}{b})  .
\end{equation}
Similarly, to   turn a counital coalgebra $M$ into a $G$-coalgebra, one needs to fix an involutive  antiautomorphism $\iota$ of $M$ such that $\varepsilon_M \iota=\varepsilon_M$. A cyclic bilinear form $v$ on $M$ is equivariant if and only if $v(\iota \otimes \iota)=v: M\otimes M\to \kk$.
By Section \ref{uniuni}, an equivariant $n$-graded  double   bracket $\double{-}{-}$  in $A$ and an equivariant  cyclic bilinear form $v$ on $M$ induce    an $n$-graded  biderivation $\bracket{-}{-}_v^{G+}$ in $A^{G+}_M$.  


Suppose   that $A=\kk[\pi]$ is the  group algebra  of a group $\pi$. We treat  $A$  as a graded algebra concentrated in degree zero and equip it with the   involutive    antiautomorphism $\iota$ inverting all elements of $\pi$. 
Consider a counital $G$-coalgebra $(M, \iota:M\to M)$ whose underlying module is free of finite rank.  For any  (non-graded) unital algebra $B$,   the convolution  algebra $H_M(B)=B\otimes M^*$ is a unital $G$-algebra  with involutive antiautomorphism $\id_B \otimes \iota^*$. We  identify the  set $\Hom_{G-\grAlg^+} (A,  H_M(B))$   with the set   of group homomorphisms   from  $ \pi   $  to     the group $U_\iota (B\otimes M^*)$ consisting of all invertible $u\in B\otimes M^*$ such that   $u^{-1} = (\id_B \otimes \iota^*) (u)$.   The bijection \eqref{eq:adjunctionequinnn} exhibits   $A^{G+}_M$ as the coordinate algebra of the functor $B\mapsto  \Hom (\pi, U_\iota (B\otimes M^*))$ where $B$ runs over unital commutative  algebras. 

For example,   consider a symmetric Frobenius algebra ${{{{\mathcal A}}}} $ equipped with  an involutive algebra automorphism   $\Delta$ preserving the Frobenius pairing.   
Pick an invertible   matrix $J\in \Mat_N({{{{\mathcal A}}}})$ with $N\geq 1$ such  that 
 $J^t=\Delta (J)$ where $t$ is the  matrix transposition and $\Delta $ applies to matrices entry-wise. (One can take $J$ to be the unit matrix or, more generally, a diagonal matrix whose   diagonal entries are invertible and fixed by $\Delta$.)  The    antiautomorphism $ X \mapsto     J\Delta (X^t) J^{-1}$  of 
  $\Mat_N({{{{\mathcal A}}}} )$   induces an involutive coalgebra  antiautomorphism $\iota$ of 
  the coalgebra  $M=(\Mat_N({{{{\mathcal A}}}}))^*$  and  makes $M$ into a counital $G$-colagebra. 
  The  cyclic   form $v_N$ on  $M $  defined in Example \ref{exam}.5 is equivariant   for any choice of $J$. This form  can be used to derive an $n$-graded  biderivation   in $A^{G+}_M$ from an equivariant double bracket in $A$.  
 Note that for a   unital  commutative algebra $B$, the group $U_\iota (B\otimes M^*)\subset GL_N(B\otimes {{{{\mathcal A}}}})$ consists  of all invertible $(N\times N)$-matrices $X$ over the ring $R_B=B\otimes {{{{\mathcal A}}}}$ such that   $  J=  X
 J \Delta (X^t ) $.  The group $U_\iota (B\otimes M^*)$  is nothing but   the group of automorphisms of the  free $R_B$-module $(R_B)^N$ of rank $N$ preserving the $R_B$-valued  sesquilinear form determined by   $J$. For $\Delta=\id_{\mathcal A}$, we recover  the group of invertible $(N\times N)$-matrices $X$ over  $R_B $ such that   $  J=  X
 J X^t   $, i.e.,   the group of automorphisms of the module  $(R_B)^N$  preserving the $R_B$-valued  bilinear form determined by   $J$.

    These constructions apply   to the fundamental group  $\pi $  of an oriented surface    with base point  in the boundary. Indeed, the   double quasi-Poisson bracket in $A=\kk[\pi]$ constructed in \cite{MT2}  satisfies   \eqref{iotaio}.   This bracket and the cyclic form $v_N$ induce    a   biderivation  in $A^{G+}_M$ which restricts to a Poisson bracket in $\mathcal E\subset A^{G+}_M$. 
As in Section~\ref{Group algebras}, the equivariant Hamiltonian reduction  yields a Poisson bracket in algebras associated with oriented surfaces with empty boundary.


\begin{thebibliography}{CJKLS}

\bibitem[1]{AKsM}
A. Alekseev, Y. Kosmann-Schwarzbach, E. Meinrenken,
\emph{Quasi-Poisson manifolds.}
Canad. J. Math. 54 (2002), no. 1, 3--29.






%
%

\bibitem[2]{Cb}
W. Crawley-Boevey,
\emph{Poisson structures on moduli spaces of representations.}
J. Algebra 325 (2011), 205--215.


 \bibitem[3]{Ko} J. Kock,   Frobenius algebras and 2D topological quantum field theories. London Math. Soc. Student Texts, 59. Cambridge Univ. Press, Cambridge, 2004.


  \bibitem[4]{BW} L. Le Bruyn, G. Van de Weyer,
\emph{Formal structures and representation spaces.}
 J. Algebra 247 (2002), no. 2, 616--635.


\bibitem[5]{MT2}
G. Massuyeau, V. Turaev,
\emph{Quasi-Poisson structures on representation spaces of surfaces.}
 Internat.  Math. Research Notices (2012) doi: 10.1093/imrn/rns215.

 \bibitem[6]{MT}
G. Massuyeau, V. Turaev,
\emph{Brackets in loop algebras of manifolds,} in preparation.

\bibitem[7]{Pr} C. Procesi,
\emph{A formal inverse to the Cayley-Hamilton theorem.}
J. Algebra 107 (1987), no. 1, 63--74.


\bibitem[8]{VdB}
M. Van den Bergh,
\emph{Double Poisson algebras.}
Trans. Amer. Math. Soc. 360 (2008), no. 11, 5711--5769.

\bibitem[9]{VdW} G. Van de Weyer,  
\emph{Double Poisson structures on finite dimensional semi-simple algebras.} Algebr. Represent. Theory 11 (2008), no. 5, 437--460. 

\end{thebibliography}
\end{document}